\documentclass[a4paper,11pt,reqno]{amsart}

\usepackage[utf8]{inputenc}
\usepackage[english]{babel}
\usepackage{amsmath}
\usepackage{amssymb}
\usepackage{amsfonts}
\usepackage{amsthm}
\usepackage{enumitem}
\usepackage{float}
\usepackage{stmaryrd}
\usepackage{graphics}
\usepackage{graphicx}
\usepackage{subfig}
\usepackage[colorlinks=false]{hyperref}
\usepackage{color}
\usepackage{tikz}
\newcommand{\wto}{\rightharpoonup}
\newcommand{\eps}{\varepsilon}
\newcommand{\mc}{\mathcal}
\newcommand{\R}{\mathbb{R}}

\newcommand{\RR}{\mc R}
\newcommand{\EE}{\mc E}
\newcommand{\FF}{\mc F}
\renewcommand{\to}{\rightarrow}
\renewcommand{\d}{\,\mathrm{d}}
\renewcommand{\div}{\mathrm{div}\,}
\DeclareMathOperator{\dist}{dist}
\DeclareMathOperator{\tr}{tr}
\DeclareMathOperator*{\argmin}{arg\,min}

\numberwithin{equation}{section}
\newtheorem{thm}{Theorem}[section]

\newtheorem{prop}[thm]{Proposition}
\newtheorem{lemma}[thm]{Lemma}
\newtheorem{cor}[thm]{Corollary}

\theoremstyle{definition}
\newtheorem{rmk}[thm]{Remark}
\newtheorem{ex}[thm]{Example}

\begin{document}
	
	\author[M.G. Mora]{Maria Giovanna Mora}
	\author[F. Riva]{Filippo Riva}
	\address[M.G. Mora and F. Riva]{Dipartimento di Matematica ``Felice Casorati'', Universit\`a di Pavia, Via Ferrata 5, 27100 Pavia (Italy)}
	\email{mariagiovanna.mora@unipv.it}
	\email{filippo.riva@unipv.it}

	\title[Pressure loads and the derivation
	of linear elasticity]{Pressure live loads and the variational derivation\\ 
		of linear elasticity}

	\begin{abstract}
		The rigorous derivation of linear elasticity from finite elasticity by means of $\Gamma$-convergence is a well-known result, which has been extended to different models also beyond the elastic regime. However, in these results the applied forces are usually assumed to be dead loads, that is, their density in the reference configuration is independent of the actual deformation. In this paper we begin a study of the variational derivation of linear elasticity in the presence of live loads.
		We consider a pure traction problem for a nonlinearly elastic body subject to a pressure live load and we compute its linearization for small pressure by $\Gamma$-convergence. We allow for a weakly coercive elastic energy density and we prove strong convergence of minimizers.
		\bigskip
		
		\bigskip
		
	\end{abstract}
	
	\maketitle
	
	\pagenumbering{arabic}

	\section{Introduction}
	
	Linear elasticity is a well-known and powerful mathematical approximation of the nonlinear theory of elasticity, with extensive application to the structural analysis and the numerical treatment of elastic bodies. In engineering textbooks its derivation is classical and is based on a formal linearization of finite elasticity about a reference configuration. 
	A rigorous mathematical derivation via $\Gamma$-convergence was developed only rather recently in the pioneering work \cite{DMNegPerc}, where a Dirichlet boundary value problem was considered.
	A similar approach was then applied to different frameworks in elasticity, such as rubber-like materials \cite{AgoDMDS}, multiwell models \cite{AliDMLazPal,AgoBlassKou, Schmidt}, 
	elasticity with residual stress \cite{ParTom,ParTom2}, and incompressible materials \cite{MaiPerc3}. 
	Beyond elasticity we also mention the papers
	\cite{Fried1,Fried2, NegToad, NegZan} for models in fracture mechanics, \cite{FriedKruz} for viscoelasticity, \cite{MielkeStef} for plasticity, and the recent contribution \cite{FriedKreZem} for materials with stress driven rearrangement instabilities.
	
	Linearization of pure traction problems has been recently studied in \cite{JesSch, MadPercTom1, MadPercTom2, MaiPerc1}, again in the context of elasticity. In this setting 
	a full $\Gamma$-convergence result has been obtained in \cite{MaorMora} and later extended to incompressible materials in \cite{MaiPerc2}.
	As observed in \cite{MaorMora}, in the Dirichlet case the boundary conditions prescribe the rigid motion to linearize about, whereas in the purely Neumann case the linearization process occurs around suitable rotations that are preferred by the applied forces.
	
	In all this literature the main focus is on understanding the behavior of the bulk elastic energy and the applied forces are usually assumed to be dead loads, namely their density in the reference configuration is independent of the actual deformation. This assumption is mathematically convenient, since the work done by the loadings turns out to be a continuous perturbation of the elastic energy, so that $\Gamma$-convergence of the total energy immediately follows from $\Gamma$-convergence of the elastic energy. 
	However, restricting the analysis to dead loads is physically unsatisfactory, since the only realistic examples of dead loads are the gravitational body force and the zero surface load (see, e.g., \cite{PodGuid, PodGuidCaff} and \cite[Section~2.7]{Ciarlet}).
	
	In this paper we begin a study of the derivation of linear elasticity in the presence of \emph{live loads}. More precisely, we consider a pure traction problem for a hyperelastic body $\Omega\subset \R^n$ subject to a (small) \emph{pressure load} on its boundary. In this setting the total energy of a deformation $y\colon \Omega\to\R^n$ is given by
	\begin{equation*}
		{\mc T}_\eps(y):=\int_\Omega W(x,\nabla y(x))\d x+\eps\int_\Omega\pi(y(x))\det\nabla y(x)\d x,
	\end{equation*}
	where the elastic energy density $W\colon \Omega\times \R^{n\times n}\to [0,+\infty]$ satisfies the usual assumptions of nonlinear elasticity (see \ref{hyp:W1}--\ref{hyp:W5}) and $\eps\pi$ is the intensity of the applied pressure load, with $\eps>0$ a small parameter and 
	$\pi\colon \R^n\to \R$ a given function. For simplicity in this introduction we assume $\pi$ to be continuous. As shown in \cite[Proposition~5.1]{PodGuidCaff} (see also \cite[Proposition~1.2.8]{KruzRoub}), the second term in the energy ${\mc T}_\eps$ is the potential of the pressure load 
	\begin{equation}\label{eq:pfield}
		-\eps \pi(y(x))({\rm cof}\,\nabla y(x))n_{\partial \Omega}(x) \qquad\text{for }x\in \partial\Omega
	\end{equation}
	acting on the whole boundary of $\Omega$, where ${\rm cof}\,F$ denotes the cofactor of the matrix $F$ and $n_{\partial\Omega}$ is the outward unit normal to $\partial\Omega$. In the deformed configuration $y(\Omega)$ the pressure load \eqref{eq:pfield} corresponds to the surface force 
	\begin{equation*}
		-\eps\pi(z)n_{\partial(y(\Omega))}(z)\qquad\text{for }z\in \partial(y(\Omega)).
	\end{equation*}
	
	Since $W(x,\cdot)$ is frame-indifferent and minimized at the identity, it is immediate to see that for $\eps=0$ the minimizers of $\mc T_\eps$ are all the rigid motions of $\Omega$.
	When $\eps$ is small, it is thus natural to expect minimizers to be close to rigid motions and their asymptotic behavior to be described by a linearization of the energy.
	In pure traction problems, as mentioned before, the applied forces select the class of rigid motions around which
	the linearization takes place (see \cite{MaorMora}). Indeed, if $y_\eps$ is a minimizer of $\mc T_\eps$, then we have
	\begin{equation*}
		{\mc T}_\eps(y_\eps)\leq \eps\int_\Omega\pi(Rx)\d x \qquad \text{for every rotation } R\in SO(n).
	\end{equation*}
	If we assume $y_\eps$ to be of the form 
	$y_\eps(x)=R_0(x+\eps u_0(x))$ with $R_0\in SO(n)$, then by a formal expansion we obtain
	\begin{equation*}
		\frac{\eps^2}{2}\int_\Omega Q(x,e(u_0)(x))\d x+\eps\int_{\Omega}\pi(R_0 x)\d x+O(\eps^2) \leq \eps\int_\Omega\pi(Rx)\d x
	\end{equation*}
	for every rotation $R\in SO(n)$. Here $Q(x,\cdot)$ is the quadratic form given by the Hessian of $W(x,\cdot)$ computed at the identity and $e(u_0)$ is the symmetric gradient of $u_0$.
	Dividing by $\eps$ and letting $\eps$ tend to zero, we deduce that $R_0$ is a so-called \emph{optimal rotation}, that is, $R_0$ belongs to the set
	\begin{equation*}
		\mc R:= \argmin\limits_{R\in SO(n)}\left\{\int_\Omega \pi(Rx)\d x\right\}.
	\end{equation*}
	
	Assume now for simplicity that the identity matrix belongs to $\RR$ (one can always reduce to this case, up to rotating the whole system). The previous argument suggests that in order to identify the limiting behavior of minimizers
	one needs to renormalize the energy as follows:
	\begin{align}
		\frac{1}{\eps^2}\mc E_\eps(y) :\! & =\frac{1}{\eps^2}\left( {\mc T}_\eps(y)-\eps\int_{\Omega}\pi(x)\d x\right) \nonumber \\
		& =
		\frac{1}{\eps^2}\int_\Omega W(x,\nabla y(x))\d x+\frac1\eps \int_\Omega \big(\pi(y(x))\det\nabla y(x)-\pi(x)\big)\d x. \label{eq:Eepsint}
	\end{align}
	
	Under suitable assumptions for $\pi$ and a weak $p$-coercivity condition on $W$ with $1<p\leq2$ (see \ref{hyp:W5}), we compute the $\Gamma$-limit of the rescaled energies $\frac1{\eps^2}\mc E_\eps$ and we establish a compactness result for deformations with equibounded energies. Here deformations are assumed to have zero average on $\Omega$, as it is common in Neumann boundary value problems.
	More precisely, we prove the following results:\medskip
	
	\noindent
	{\bf Compactness:} If $\mc E_\eps(y_\eps)\leq C\eps^2$, then there exist rotations $R_\eps\in SO(n)$ and displacements $u_\eps\in W^{1,p}(\Omega;\R^n)$
	such that 
	\begin{equation}\label{eq:yeps2}
		y_\eps(x)=R_\eps(x+\eps u_\eps(x)) \qquad \text{for } x\in\Omega
	\end{equation}
	and, up to subsequences, there holds\smallskip
	\begin{itemize}
		\item[$\bullet$] $u_\eps\wto u_0$ weakly in $W^{1,p}(\Omega;\R^n)$ with $u_0\in H^{1}(\Omega;\R^n)$,\smallskip
		\item[$\bullet$] $R_\eps\to R_0$ with $R_0\in \RR$.\medskip
	\end{itemize}
	
	\noindent
	{\bf $\Gamma$-convergence:} Under the above notion of convergence $y_\eps\to (u_0,R_0)$, the rescaled energies $\frac{1}{\eps^2}\mc E_\eps$ $\Gamma$-converge to
	\begin{equation*}
		\mc E_0(u_0,R_0):=\frac 12\int_\Omega Q(x,e(u_0)(x))\d x+\int_{\partial\Omega}\pi(R_0x)n_{\partial\Omega}(x)\cdot u_0(x)\d\mc H^{n-1}(x).\medskip
	\end{equation*}
	
	We also deduce \emph{strong convergence} of (almost) minimizers: if $(y_\eps)$ is a sequence of (almost) minimizers, then we have in addition that $u_\eps\to u_0$ strongly in $W^{1,p}(\Omega;\R^n)$ and the pair $(u_0,R_0)$ is a minimizer of $\mc E_0$.
	
	We now comment on the expression of the limiting energy $\mc E_0$, which features two terms: the usual linear elastic energy and a potential term accounting for the surface load $-\pi(R_0x)n_{\partial\Omega}(x)$ on $\partial\Omega$. The emergence of this boundary term can be explained by the following heuristic considerations: on the one hand, a formal linearization of \eqref{eq:pfield} leads to a pressure load of the above form; on the other hand, if $y$ is smooth enough, the force term in \eqref{eq:Eepsint} can be written as 
	\begin{equation*}
		\frac1\eps \left( \int_{y(\Omega)}\pi(x)\d x -\int_{\Omega}\pi(x)\d x\right) .
	\end{equation*}
	Hence, taking into account \eqref{eq:yeps2}, computing the limit of the above expression on sequences $(y_\eps)$ with equibounded energies corresponds to a sort of shape derivative of the functional $\Omega\mapsto \int_{\Omega}\pi(x)\d x$ (see, e.g., \cite[Proposition~17.8]{Maggi}). 
	However, we stress that in the present setting deformations $y_\eps$ are only of Sobolev regularity and are close to rigid motions only in the sense of $W^{1,p}(\Omega;\R^n)$, therefore the usual arguments in the context of shape derivatives do not apply.
	
	From a mathematical viewpoint the main difference with respect to previous contributions dealing with dead loads, is that the force term in \eqref{eq:Eepsint} is not a continuous perturbation of the elastic energy.
	Indeed, our assumptions on $W$ imply that deformations are at most strongly convergent in $W^{1,p}(\Omega;\R^n)$ with $1<p\leq 2$ and this is not enough to guarantee convergence of the determinants. Moreover, the crucial step in the proof of compactness is to show that deformations satisfying the bound $\mc E_\eps(y_\eps)\leq C\eps^2$ have an elastic energy of order $\eps^2$. Once this is established, one can apply the rigidity estimate by Friesecke, James, and M\"uller \cite{FriesJamesMul}
	and deduce \eqref{eq:yeps2}, together with a uniform bound for $(u_\eps)$ in $W^{1,p}(\Omega;\R^n)$. In the case of dead loads deducing the $\eps^2$-bound on the elastic energy is straightforward, since the force term is linear with respect to the deformation. In our setting, instead, this is one of the main difficulties. We show that the problem can be solved under two different sets of conditions: 
	\begin{itemize}
		\item $\pi$ Lipschitz continuous in a suitable neighborhood of $\Omega$ and nonnegative;
		\item $\pi$ Lipschitz continuous in a suitable neighborhood of $\Omega$, with a growth condition on its negative part (see \ref{hyp:pi3}) and an additional coercivity property for $W(\cdot, F)$ in terms of $\det F$ (see \ref{hyp:W6}).
	\end{itemize}
	We note that this additional coercivity condition on $W$ is satisfied by a large class of elastic materials (see, e.g., \cite[Remark~2.8]{AgoDMDS}).
	
	We also observe that both the nonlinear energy \eqref{eq:Eepsint} and the $\Gamma$-limit $\mc E_0$ are well defined if $\pi$ is merely a continuous function. However, because of the low regularity of deformations, in our proofs we need $\pi$ to be Lipschitz continuous, at least in a suitable neighborhood of $\Omega$. How to extend our analysis to less regular pressure loads is an interesting question that we plan to consider in a future work.
	
	Finally, in \cite{MaorMora} it is proved that, in the case of dead loads, the set $\RR$ of optimal rotations is a submanifold of $SO(n)$ and, as a consequence, one can prove that the distance of the approximating rotations $R_\eps$ in \eqref{eq:yeps2} from $\RR$ is at most of order $\sqrt\eps$. In the last part of the paper we show that
	neither of these properties is true, in general, in the present setting.

	\subsection*{Plan of the paper}
	In Section~\ref{sec:setting} we set the problem and we state the main assumptions. In Section~\ref{sec:nonnegative} we discuss the case of a nonnegative pressure intensity $\pi$. We then extend our analysis to pressures with arbitrary sign in Section~\ref{sec:general}. Finally, in Section~\ref{sec:secondorder} we compute a refined $\Gamma$-limit, which takes into account how much deformations differ from being optimal rotations, and we make a comparison with the results proved in \cite{MaorMora} in the case of dead loads.
	
	\section{Setting of the problem}\label{sec:setting}
	
	\subsection{Notation and preliminaries} 
	Throughout the paper, the symbols $C$ or $c$ will be used to denote some positive constants not depending of $\eps$, whose value may change from line to line.\par 
	Given two (extended) real numbers $a$ and $b$ the notation $a\vee b$ (respectively, $a\wedge b$) stands for the maximum (respectively, the minimum) between the two numbers. Given a scalar function $f$, we denote its positive and negative part by $f^+$ and $f^-$, respectively, so that $f=f^+-f^-$. By $B_r\subset \R^n$ we mean the open ball with radius $r>0$ centered at the origin. 
	
	Let $\Omega$ be an open set in $\R^n$. For $p\in [1,\infty]$ 
	the norms in $L^p(\Omega)$ and $L^p(\Omega;\R^n)$ will be simply denoted by $\|\cdot\|_p$. The conjugate exponent of $p\in[1,\infty]$ will be denoted by $p'$. 
	The notation $\mathring{W}^{1,p}(\Omega;\R^n)$ stands for the space of Sobolev functions $y\in {W}^{1,p}(\Omega;\R^n)$ with zero average; if $p=2$, we shall write $\mathring{H}^{1}(\Omega;\R^n)$ instead of $\mathring{W}^{1,2}(\Omega;\R^n)$. 
	
	We denote by $\R^{n\times n}$, $\R^{n\times n}_{\rm sym}$, and $\R^{n\times n}_{\rm skew}$ the set of $(n\times n)$-matrices and the subsets of symmetric and skew-symmetric matrices, respectively. The set of rotations is denoted by $SO(n)$, namely
	\begin{equation*}
		SO(n)=\{R\in\R^{n\times n}: \ R^TR=I, \ \det R=1\}.
	\end{equation*}
	Finally, we recall that for every $F\in \R^{n\times n}$ and $\eps>0$ there holds
	\begin{equation}\label{eq:sviluppodet}
		\det(I+\eps F)=1+\sum_{k=1}^{n}\eps^k \iota_k(F),
	\end{equation}
	where $\iota_k(F)$ is a homogeneous polynomial of degree $k$ in the entries of $F$. In particular, there exists a constant $C>0$, depending only on $n$,
	such that
	\begin{equation}\label{eq:estik}
		|\iota_k(F)|\le C|F|^k \qquad\text{for every } F\in \R^{n\times n}.
	\end{equation}
	For $k=1$ and $k=n$ we have that $\iota_1(F)=\tr F$ and $\iota_n(F)=\det F$.
	
	For the definition and the properties of $\Gamma$-convergence we refer to the monograph \cite{DM}.
	
	\subsection{The main assumptions}
	Let $\Omega\subset\R^n$, with $n\ge 2$, be a bounded domain with Lipschitz boundary representing the reference configuration of a hyperelastic body. Up to a translation of the axes, we can assume without loss of generality that the origin is the barycenter of $\Omega$, i.e.,
	\begin{equation}\label{eq:baricenter}
		\int_{\Omega}x\d x=0.
	\end{equation}
	For future use we also introduce the set
	\begin{equation*}
		\mc O:=\bigcup_{R\in SO(n)}R\Omega,
	\end{equation*}
	which is an open annulus (if $0\not\in\Omega$; an open ball if $0\in\Omega$) centered at $0$ and containing $\Omega$.
	
	The stored energy density of the body is assumed to be a Carath\'eodory function $W\colon \Omega\times \R^{n\times n}\to [0,+\infty]$ 
	satisfying the following conditions for almost every $x\in \Omega$:\smallskip
	\begin{enumerate}[label=\textup{(W\arabic*)}]
		\item\label{hyp:W1} $W(x,F)=+\infty$ if $\det F\leq 0$ (orientation preserving condition);\smallskip
		
		\item\label{hyp:W2} $W(x,RF)=W(x,F)$ for every $F\in \R^{n\times n}$ and $R\in SO(n)$ (frame indifference);\smallskip
		
		\item\label{hyp:W3} $W(x,I)=0$ (the reference configuration is stress-free);\smallskip
		
		\item\label{hyp:W4} $W(x,\cdot)$ is of class $C^2$ in a neighborhood of $SO(n)$, independent of $x$, where the second derivatives of $W$ are bounded, 
		uniformly with respect to $x\in\Omega$;\smallskip
		
		\item\label{hyp:W5} $W(x,F)\ge c_1 g_p(\dist(F;SO(n)))$ for every $F\in \R^{n\times n}$ and for some $p\in(1,2]$,
		where $g_p$ is defined as
		\begin{equation}\label{eq:gp}
			g_p(t):=\begin{cases}
				\frac{t^2}{2}&\text{ if }t\in [0,1],\smallskip\\
				\frac{t^p}{p}+\frac 12-\frac 1p&\text{ if }t>1,
			\end{cases}
		\end{equation}
		and $c_1>0$ is a constant independent of $x$ (coercivity).\smallskip
	\end{enumerate}
	Assumptions \ref{hyp:W1}--\ref{hyp:W3} are natural conditions in elasticity theory (see, e.g., \cite{Ciarlet,KruzRoub}), assumption \ref{hyp:W4} is the minimal regularity hypothesis needed to perform the linearization, while condition \ref{hyp:W5} is satisfied by a large class of compressible rubber-like materials (see, e.g., \cite{AgoBlassKou,AgoDMDS,MaiPerc1,MaiPerc2,MaiPerc3}).
	
	We note that 
	\begin{equation}\label{eq:propgp}
		g_p(t)\ge \frac 12 (t^2\wedge t^p)\qquad\text{for every }t\ge 0.
	\end{equation}
	Moreover, condition \ref{hyp:W5} implies the following bound:
	\begin{equation}\label{eq:coerc2}
		W(x,F)\ge c\,|\det F-1|^2\quad\text{for a.e.\ }x\in \Omega\text{ and for every }F\in \R^{n\times n} \text{ with }|\det F-1|\le 1,
	\end{equation}
	for a suitable constant $c>0$, independent of $x$. Indeed, by \ref{hyp:W5} the bound \eqref{eq:coerc2} is satisfied for $F$ outside a neighborhood of $SO(n)$.
	If instead $\dist(F;SO(n))$ is small enough, then one has
	\begin{equation*}
		c|\det F-1|^2\leq \,\dist^2(F;SO(n)),
	\end{equation*}
	which implies \eqref{eq:coerc2} by using again \ref{hyp:W5}.
	
	We assume the body to be subjected to a \emph{pressure load}, whose (unscaled) intensity is a Borel measurable function $\pi\colon \R^n\to \R$ such that\smallskip
	\begin{enumerate}[label=\textup{($\pi$\arabic*)}]
		\item\label{hyp:pi1} $\pi$ is Lipschitz continuous in an open set containing $\overline{\mc O}$.\smallskip
	\end{enumerate}
	Since we do not prescribe any Dirichlet boundary condition, the linearization process will naturally select, as in \cite{MaorMora} in the case of dead loads, a particular set of rotations that are ``preferred'' by the force. This set is called the set of \emph{optimal rotations} and in our framework it is defined as
	\begin{equation}\label{eq:optrot}
		\RR:=\argmin\limits_{R\in SO(n)}\left\{\int_\Omega \pi(Rx)\d x\right\}.
	\end{equation} 
	Since the map $R\mapsto \int_\Omega \pi(Rx)\d x$ is continuous and $SO(n)$ is compact, the set of optimal rotations is not empty and is a compact subset of $SO(n)$. 
	For simplicity we assume that 
	\begin{equation}\label{eq:IinR}
		I\in\RR,
	\end{equation}
	where $I$ is the identity matrix. Indeed, if this is not the case, we can always replace $\pi$ by $\pi(R_0\cdot)$ and deformations $y$ by $R_0^Ty$, where $R_0$ is a given optimal rotation.
	
	Let $R_0\in \RR$. By computing the first variation of the functional in \eqref{eq:optrot} along the curve $t\mapsto R_0e^{tA}$ with $A\in \R^{n\times n}_{\rm skew}$,
	we deduce that any optimal rotation $R_0$ satisfies the following Euler-Lagrange equation:
	\begin{equation}\label{eq:EL}
		\int_{\Omega}\nabla\pi (R_0x)\cdot R_0Ax\d x=0\qquad\text{for every }A\in \R^{n\times n}_{\rm skew}.
	\end{equation}
	Applying the Divergence Theorem, condition \eqref{eq:EL} can be rewritten as
	\begin{equation}\label{eq:EL2}
		\int_{\partial\Omega} \pi(R_0x)n_{\partial\Omega}(x)\cdot Ax\d\mc H^{n-1}(x)=0\qquad\text{for every }A\in \R^{n\times n}_{\rm skew},
	\end{equation}
	where $n_{\partial\Omega}$ is the outward unit normal to $\partial\Omega$.

	\section{Nonnegative pressure loads}\label{sec:nonnegative}
	
	We start our analysis by considering a pressure load with nonnegative intensity, that is,\smallskip
	\begin{enumerate}[label=\textup{($\pi$2)}]
		\item\label{hyp:pi2} $\pi(y)\ge 0$ for every $y\in \R^n$.\smallskip
	\end{enumerate}
	This includes, for instance, the relevant case of hydrostatic pressure $\pi(y)=g\rho y_3^-$, where $g$ is the gravitational constant, $\rho$ is the constant density of the fluid, and
	$y_3^-$ denote the negative part of the third component of $y$.
	
	For every $\eps\in(0,1)$ we consider the energy $\EE_\eps\colon \mathring{W}^{1,p}(\Omega;\R^n)\to (-\infty,+\infty]$ defined as
	\begin{equation}\label{eq:energy1}
		\EE_\eps(y):=\begin{cases}\displaystyle
			\int_\Omega W(x,\nabla y(x))\d x+\eps\int_\Omega\left(\pi(y(x))\det\nabla y(x)-\pi(x)\right)\d x&\text{ if }y\in Y^p,\\
			+\infty&\text{ otherwise,}
		\end{cases}
	\end{equation}
	where the set of admissible deformations is
	\begin{equation*}
		Y^p:=\left\{y\in \mathring{W}^{1,p}(\Omega;\R^n):\ \det\nabla y(x)>0 \text{ for a.e.\ }x\in \Omega \right\}.
	\end{equation*} 
	In other words, admissible deformations are orientation preserving and, as it is common in Neumann boundary value problems, have zero average on $\Omega$.
	
	By \eqref{eq:baricenter} any rigid motion of the form 
	\begin{equation}\label{eq:defyR}
		y_R(x):=Rx \quad \text{with }R\in SO(n)
	\end{equation} 
	belongs to $Y^p$. 
	Moreover, under the assumption \ref{hyp:pi2}, the energy is well defined since the two integrands $W(\cdot, \nabla y)$ and
	$\pi(y)\det\nabla y$ are nonnegative for~$y\in Y^p$.
	
	\begin{rmk}\label{rmk:injective}
		Here we do not assume deformations to be injective in any sense. However, one can easily include the requirement that admissible deformations are a.e.\ injective (see, e.g., \cite{Ciarlet}), without affecting the results of the paper, see also Remark~\ref{rmk:injectiveinf}. 
	\end{rmk}
	
	The key ingredient in the proof of compactness is the following variant of the celebrated \emph{rigidity estimate} by Friesecke, James, and M\"uller \cite{FriesJamesMul}, whose proof can be found, e.g., in~\cite[Lemma~3.1]{AgoDMDS}. Similar variants of the rigidity estimates with mixed growth condition have been proved in \cite{CD,MPal,SZ}.
	
	\begin{thm}\label{lemma:prigidity}
		There exists a positive constant $C=C(\Omega,p)>0$ with the following property: for every $y\in W^{1,p}(\Omega;\R^n)$ there exists a constant rotation $R\in SO(n)$ such that
		\begin{equation*}
			\int_\Omega g_p(|\nabla y(x)-R|)\d x\le C\int_\Omega g_p(\dist(\nabla y(x);SO(n)))\d x.
		\end{equation*}
	\end{thm}
	
	The following generalized rigidity estimate will be used in Theorem~\ref{thm:convmin} to infer strong convergence of almost minimizers.
	For a proof we refer to~\cite[Theorem~1.1]{CDM}.
	
	\begin{thm}\label{lemma:rigidity2}
		Let $1<p_1<p_2<\infty$. Then there exists a positive constant $C=C(\Omega,p_1,p_2)>0$ with the following property: for every $y\in W^{1,1}(\Omega;\R^n)$ with 
		\begin{equation*}
			\dist(\nabla y;SO(n))=f_1+f_2\qquad\text{for some }f_i\in L^{p_i}(\Omega),\,\, i=1,2,
		\end{equation*}
		there exist a constant rotation $\widetilde{R}\in SO(n)$ and two functions $g_i\in  L^{p_i}(\Omega)$ such that
		\begin{equation*}
			\nabla y=\widetilde{R}+g_1+g_2\qquad\text{ and }\qquad\|g_i\|_{p_i}\le C\|f_i\|_{p_i},\,\, i=1,2.
		\end{equation*}
	\end{thm}
	
	Our arguments will strongly rely on the Lipschitz continuity of the pressure function $\pi$. This, however, holds only in a suitable set $\Omega'$ containing $\Omega$ (see \ref{hyp:pi1}). Since deformations $y\in Y^p$ may be a priori valued outside $\Omega'$, it is convenient to introduce an auxiliary Lipschitz continuous function that coincides with $\pi$ in $\Omega'$ and is bounded on the whole of $\R^n$. This is the content of the following lemma, which is clearly not necessary if $\pi$ itself is Lipschitz continuous and bounded.
	
	\begin{lemma}\label{lemma:hatpi1}
		Assume \ref{hyp:pi1} and \ref{hyp:pi2}.
		Then there exists a Lipschitz continuous function $\hat\pi\colon \R^n\to [0,+\infty)$, with compact support, such that $\hat\pi$ coincides with $\pi$ in an open neighborhood of $\overline{\mc O}$ and 
		$\hat\pi(y)\le\pi(y)$ for all $y\in \R^n$. In particular, $\hat{\pi}$ is bounded.	
		\end{lemma}
	
	\begin{rmk}
		Note that the set of optimal rotations \eqref{eq:optrot} stays the same if $\pi$ is replaced by $\hat\pi$, since $\pi$ and $\hat\pi$ coincide in a neighborhood of $\overline{\mc O}$.
	\end{rmk}
	
	\begin{proof}[Proof of Lemma~\ref{lemma:hatpi1}]
		Assume $0\not\in\Omega$, so that $\mathcal O$ is an open annulus centered at $0$ (the case where $0\in\Omega$ can be treated similarly).
		Let $0<r_1<r_2$ and $0<\delta<r_1$ be such that $\overline{\mc O}\subset B_{r_2}\setminus\overline{B_{r_1}}$ and $\pi$ is Lipschitz continuous in $\overline{B_{r_2+\delta}}\setminus B_{r_1-\delta}$ with Lipschitz constant $L$.
		Let $M\geq0$ be the maximum of $\pi$ on $\partial B_{r_1}\cup\partial B_{r_2}$.
		We first define $\widetilde{\pi}: \overline{B_{r_2+\delta}}\setminus B_{r_1-\delta}\to\R$ as
		\begin{equation*}
			\widetilde{\pi}(y):=\begin{cases}
				\pi\left(r_1\frac{y}{|y|}\right)-K(r_1-|y|) & \text{ if }r_1-\delta\leq |y|<r_1,\smallskip\\
				\pi(y) & \text{ if } r_1\leq |y|<r_2,\smallskip\\
				\pi\left(r_2\frac{y}{|y|}\right)-K(|y|-r_2) & \text{ if } r_2\leq |y|\leq r_2+\delta,
			\end{cases}
		\end{equation*}
		where $K:=L\vee\frac{M}{\delta}$.
		It is easy to see that $\widetilde{\pi}$ is Lipschitz continuous in its domain; moreover, by construction it coincides with $\pi$ in an open neighborhood of $\overline{\mc O}$. We now show  that $\widetilde{\pi}\le \pi$ on $\overline{B_{r_2+\delta}}\setminus B_{r_1-\delta}$.
		Indeed, if $r_1-\delta\leq |y|<r_1$, by the Lipschitz continuity of $\pi$ we have
		\begin{equation*}
			\widetilde{\pi}(y)\le \pi(y)+L\left|r_1 \frac{y}{|y|}-y\right|-K(r_1-|y|)=\pi(y)-(K-L)(r_1-|y|)\le \pi(y),
		\end{equation*}
		and similarly if $r_2\leq |y|\leq r_2+\delta$. Finally, we note that, if $y\in\partial  B_{r_1-\delta}\cup\partial B_{r_2+\delta}$, then
		\begin{equation*}
			\widetilde{\pi}(y)\le M-K\delta\le 0.
		\end{equation*}
		We now conclude by considering $\hat \pi(y):=\widetilde{\pi}(y)\vee 0$ for $y\in \overline{B_{r_2+\delta}}\setminus B_{r_1-\delta}$, $\hat \pi(y):=0$ otherwise in~$\R^n$.
	\end{proof}
	
	We are now in a position to state and prove some estimates which will be crucial to infer compactness of deformations with equibounded (rescaled) energy.
	
	\begin{lemma}\label{lemma:est1}
		Assume \ref{hyp:W1}--\ref{hyp:W5}, \ref{hyp:pi1}, and \ref{hyp:pi2}. If $\EE_\eps(y_\eps)\le C\eps^2$ for every $\eps\in(0,1)$, then there holds
		\begin{equation}\label{eq:estdet}
			\int_{\{|\det\nabla y_\eps-1|\le 1\}}|\det\nabla y_\eps(x)-1|^2\d x\le C\eps^2 \quad \text{for every }\eps\in(0,1).
		\end{equation}
		Furthermore, there exist constant rotations $R_\eps\in SO(n)$ such that the rescaled displacements $u_\eps\in \mathring{W}^{1,p}(\Omega;\R^n)$ defined by
		\begin{equation}\label{eq:displacement}
			u_\eps:=\frac 1\eps R_\eps^T(y_\eps-y_{R_\eps})
		\end{equation}
		(see \eqref{eq:defyR} for the definition of $y_{R_\eps}$) satisfy
		\begin{equation}\label{eq:estgp}
			\int_{\Omega}g_p(\eps|\nabla u_\eps(x)|)\d x\le C\eps^2  \quad \text{for every }\eps\in(0,1)
		\end{equation}
		and are uniformly bounded in $W^{1,p}(\Omega;\R^n)$.
		
		If, moreover, $(R'_\eps)\subset SO(n)$ is another sequence for which the rescaled displacements, defined as in \eqref{eq:displacement}, satisfy \eqref{eq:estgp}, then
		\begin{equation}\label{eq:estuniq}
			|R_\eps-R'_\eps|\le C\eps
		\end{equation} 
		for every $\eps\in(0,1)$.
	\end{lemma}
	
	\begin{proof}
		Let $\widehat{\mc E}_\eps$ be the auxiliary energy defined as in \eqref{eq:energy1} with $\pi$ replaced by the function $\hat\pi$ given by Lemma~\ref{lemma:hatpi1}. 
		Since $\pi\ge\hat\pi$ everywhere and $\pi\equiv\hat\pi$ on $\Omega$, we have that
		\begin{equation}\label{eq:Ehat}
			\widehat{\mc E}_\eps(y)\le \mc E_\eps(y) \qquad\text{for every } y\in\mathring{W}^{1,p}(\Omega;\R^n).
		\end{equation}
		For the sake of brevity we introduce the notation
		\begin{equation}\label{eq:omegapiumeno}
			\Omega_\eps^-:=\left\{x\in \Omega: \ |\det\nabla y_\eps(x)-1|\le 1\right\},\qquad \Omega_\eps^+:=\Omega\setminus \Omega_\eps^-.	
		\end{equation}
		Since $\widehat{\mc E}_\eps(y_\eps)\le\EE_\eps(y_\eps)\le C\eps^2$, we have that $y_\eps$ belongs to $Y^p$, hence in particular $\det\nabla y_\eps>0$ a.e.\ in $\Omega$. By \ref{hyp:W5} and Theorem~\ref{lemma:prigidity} we infer the existence of $R_\eps\in SO(n)$ such that
		\begin{equation}\label{eq:gpW}
			\int_\Omega g_p(|\nabla y_\eps(x)-R_\eps|)\d x\le C\int_\Omega W(x,\nabla y_\eps(x))\d x.
		\end{equation}
		By \eqref{eq:IinR} we deduce that
		\begin{align*}
			\int_\Omega W(x,\nabla y_\eps)\d x&=\widehat{\mc E}_\eps(y_\eps)+\eps\int_\Omega\left(\hat\pi(x)-\hat\pi(y_\eps)\det\nabla y_\eps\right)\d x\\
			&\le C\eps^2+\eps\int_\Omega\left(\hat\pi(y_{R_\eps})-\hat\pi(y_\eps)\det\nabla y_\eps\right)\d x\\
			&\le C\eps^2+\eps\int_\Omega|\hat\pi(y_\eps)-\hat\pi(y_{R_\eps})|\d x+\eps\int_\Omega\hat\pi (y_\eps)(1-\det\nabla y_\eps)\d x.
		\end{align*}
		Since $\hat\pi\geq0$ by construction, the integrand in the last integral above is nonpositive on $\Omega_\eps^+$. Thus, using the fact that $\hat\pi$ is Lipschitz continuous and bounded, and applying H\"older's inequality we deduce that 
		\begin{equation}\label{eq:West}
			\int_\Omega W(x,\nabla y_\eps)\d x\le C\eps^2+C\eps\|y_\eps-y_{R_\eps}\|_1+C\eps \|\det\nabla y_\eps-1\|_{L^2(\Omega_\eps^-)}.
		\end{equation}
		By \eqref{eq:coerc2} this implies
		\begin{equation*}
			\|\det\nabla y_\eps-1\|_{L^2(\Omega_\eps^-)}^2\le C\eps^2+C\eps\|y_\eps-y_{R_\eps}\|_1+C\eps \|\det\nabla y_\eps-1\|_{L^2(\Omega_\eps^-)},
		\end{equation*}
		which, in turn, by Young's inequality yields
		\begin{equation}\label{eq:det}
			\|\det\nabla y_\eps-1\|_{L^2(\Omega_\eps^-)}^2\le C\eps^2+C\eps\|y_\eps-y_{R_\eps}\|_1.
		\end{equation}
		By combining \eqref{eq:West} and \eqref{eq:det} we deduce
		\begin{equation}\label{eq:WL1}
			\int_\Omega W(x,\nabla y_\eps)\d x\le C\eps^2+C\eps\|y_\eps-y_{R_\eps}\|_1,
		\end{equation}
		and, as a consequence of \eqref{eq:displacement}, \eqref{eq:gpW}, and \eqref{eq:WL1}, we obtain
		\begin{equation}\label{eq:gpbound}
			\int_\Omega g_p(|\eps\nabla u_\eps|)\d x=\int_\Omega g_p(|\nabla y_\eps-R_\eps|)\d x\le C\eps^2+C\eps\|y_\eps-y_{R_\eps}\|_1=C\eps^2(1+\|u_\eps\|_1).
		\end{equation}
		Using the definition \eqref{eq:gp} of $g_p$ this implies that
		\begin{equation}\label{eq:L2bound}
			\int_{\{|\eps\nabla u_\eps|\le 1\}}|\eps\nabla u_\eps|^2\d x\le 2 \int_\Omega g_p(|\eps\nabla u_\eps|)\d x\le C\eps^2(1+\|u_\eps\|_1).
		\end{equation}
		By H\"older's inequality we obtain
		\begin{equation}\label{eq:Lpbound1}
			\int_{\{|\eps\nabla u_\eps|\le 1\}}|\eps\nabla u_\eps|^p\d x\le C\left(\int_{\{|\eps\nabla u_\eps|\le 1\}}|\eps\nabla u_\eps|^2\d x\right)^\frac{p}{2}\le C\eps^p(1+\|u_\eps\|_1)^\frac{p}{2}\le C\eps^p(1+\|u_\eps\|_1),
		\end{equation}
		where the last inequality follows from the fact that $t^\frac{p}{2}\le 1+t$ for $t\ge 0$.
		
		Again from \eqref{eq:gp} and recalling that $p\le 2$ we also have that
		\begin{equation}\label{eq:Lpbound2}
			\int_{\{|\eps\nabla u_\eps|> 1\}}|\eps\nabla u_\eps|^p\d x\le\int_\Omega g_p(|\eps\nabla u_\eps|)\d x\le C\eps^2(1+\|u_\eps\|_1)\le C\eps^p(1+\|u_\eps\|_1).
		\end{equation}
		By \eqref{eq:Lpbound1}, \eqref{eq:Lpbound2}, and the continuous embedding of $W^{1,p}(\Omega;\R^n)$ into $L^1(\Omega;\R^n)$ we deduce that
		\begin{equation*}
			\|\nabla u_\eps\|_p^p\le C+C\|u_\eps\|_{W^{1,p}}.
		\end{equation*}
		Since $u_\eps$ has zero average, Poincar\'e-Wirtinger inequality finally yields
		\begin{equation*}
			\|u_\eps\|_{W^{1,p}}^p\le C+C\|u_\eps\|_{W^{1,p}},
		\end{equation*}
		which implies $\|u_\eps\|_{W^{1,p}}\le C$. This inequality, combined with \eqref{eq:det} and \eqref{eq:gpbound}, provides \eqref{eq:estdet} and \eqref{eq:estgp}.
		
		Finally, if $(R_\eps')$ is a sequence of rotations whose corresponding rescaled displacements satisfy \eqref{eq:estgp}, then 
		\begin{equation*}
			|R_\eps-R_\eps'|\le C(\|\nabla y_\eps-R_\eps\|_p+\|\nabla y_\eps-R_\eps'\|_p)\le C\eps.
		\end{equation*}
		This concludes the proof.
	\end{proof}
	
	As an immediate corollary we obtain that the infimum of the energy $\mc E_\eps$ is of order $\eps^2$.
	
	\begin{cor}\label{cor:inf}
		Assume \ref{hyp:W1}--\ref{hyp:W5}, \ref{hyp:pi1}, and \ref{hyp:pi2}. Then 
		\begin{equation}\label{eq:infest}
			-C\eps^2\le\inf\limits_{\mathring{W}^{1,p}(\Omega;\R^n)}\EE_\eps\le 0  \quad \text{for every }\eps\in(0,1).
		\end{equation}
	\end{cor}
	
	\begin{proof}
		Let $(y_\eps)$ be a minimizing sequence satisfying
		\begin{equation*}
			\EE_\eps(y_\eps)\le\inf\limits_{\mathring{W}^{1,p}(\Omega;\R^n)}\EE_\eps +\eps^2.
		\end{equation*}
		Using the fact that $W$ is nonnegative and arguing as in the proof of Lemma~\ref{lemma:est1}, we deduce that
		\begin{align*}
			\EE_\eps(y_\eps)&\ge \eps\int_\Omega\left(\hat\pi (y_\eps)\det\nabla y_\eps-\hat\pi(y_{R_\eps})\right)\d x\ge -C\eps \|y_\eps-y_{R_\eps}\|_1-C\eps\|\det\nabla y_\eps-1\|_{L^2(\Omega_\eps^-)}\\
			&\ge -C\eps^2,
		\end{align*}
		where the last inequality follows from Lemma~\ref{lemma:est1}. This proves the first inequality in \eqref{eq:infest}.
		
		The other inequality in \eqref{eq:infest} follows trivially by the fact that the energy $\EE_\eps$ is zero on the identity map.
	\end{proof}
	
	We now have all the main ingredients to prove compactness of deformations with equibounded rescaled energies.
	
	\begin{prop}[\textbf{Compactness}]\label{prop:compactness}
		Assume \ref{hyp:W1}--\ref{hyp:W5}, \ref{hyp:pi1}, and \ref{hyp:pi2}. If $\EE_\eps(y_\eps)\le C\eps^2$ for every $\eps\in(0,1)$, then for any $R_\eps$, $u_\eps$ given by Lemma~\ref{lemma:est1} we have that, up to subsequences,
		\begin{itemize}
			\item $u_\eps\wto u_0$ weakly in $\mathring{W}^{1,p}(\Omega;\R^n)$ with $u_0\in \mathring{H}^{1}(\Omega;\R^n)$,\smallskip
			\item $R_\eps\to R_0$ with $R_0\in \RR$,
		\end{itemize}
		as $\eps\to 0$. Moreover, $R_0$ is independent of the choice of $R_\eps$ 
		and $u_0$ is independent up to infinitesimal rigid motions of the form $Ax$, with $A\in \R^{n\times n}_{\rm skew}$.
	\end{prop}
	
	\begin{proof}
		By Lemma~\ref{lemma:est1} the sequence $(u_\eps)$ is uniformly bounded in $W^{1,p}(\Omega;\R^n)$. Hence, up to subsequences, $u_\eps\wto u_0$ weakly in $\mathring{W}^{1,p}(\Omega;\R^n)$. We now show that $u_0$ belongs to $\mathring{H}^{1}(\Omega;\R^n)$.
		We first introduce the set 
		\begin{equation}\label{eq:goodset}
			G_\eps:=\left\{x\in \Omega: \ \eps^\frac 12|\nabla u_\eps(x)|\le 1\right\},
		\end{equation}
		and we observe that by Tchebichev inequality 
		\begin{equation}\label{eq:goodsetvanish}
			|\Omega\setminus G_\eps|\le C\eps^\frac p2.
		\end{equation}
		We claim that
		\begin{itemize}
			\item[(i)] $\chi_{G_\eps}\nabla u_\eps$ is bounded in $L^2(\Omega;\R^{n\times n})$;\smallskip
			\item[(ii)] $\nabla u_0\in L^2(\Omega;\R^{n\times n})$ and, up to subsequences, $\chi_{G_\eps}\nabla u_\eps\wto \nabla u_0$ weakly in $L^2(\Omega;\R^{n\times n})$.
		\end{itemize}
		Assertion (i) easily follows from \eqref{eq:estgp} arguing as in \eqref{eq:L2bound} and using that $G_\eps\subset \{|\eps\nabla u_\eps|\le 1\}$. To prove (ii) we first note that (i) ensures that $\chi_{G_\eps}\nabla u_\eps \wto v$ weakly in $L^2(\Omega;\R^{n\times n})$, up to subsequences, for some $v\in L^2(\Omega;\R^{n\times n})$. On the other hand, by \eqref{eq:goodsetvanish} we have that 
		$\chi_{G_\eps}$ converges to $1$ boundedly in measure. Since $\nabla u_\eps\wto \nabla u_0$ in $L^p(\Omega;\R^{n\times n})$, we conclude that $ \chi_{G_\eps}\nabla u_\eps\wto \nabla u_0$ in $L^p(\Omega;\R^{n\times n})$. Hence, $v=\nabla u_0$ and (ii) is proved. By Sobolev embedding we have that $u_0\in \mathring{H}^1(\Omega;\R^n)$.
		
		Since $SO(n)$ is a compact set, there exists $R_0\in SO(n)$ such that $R_\eps\to R_0$, up to subsequences. To prove that $R_0\in\RR$ we argue as in the proof of Lemma~\ref{lemma:est1} and deduce
		\begin{equation*}
			C\ge \frac{1}{\eps^2}\EE_\eps(y_\eps)\ge \frac{1}{\eps^2}\widehat{\EE}_\eps(y_\eps)\ge \frac 1\eps\int_\Omega(\hat\pi(y_\eps)\det\nabla y_\eps-\hat\pi(x))\d x\ge -c+\frac 1\eps\int_\Omega(\pi(R_\eps x)-\pi(x))\d x.
		\end{equation*}
		Note that in the last integral we used that $\hat\pi\equiv\pi$ on $\mc O$.
		
		By multiplying by $\eps$ and then letting $\eps\to 0$ we infer that
		\begin{equation*}
			\int_{\Omega}(\pi(R_0 x)-\pi(x))\d x\le 0.
		\end{equation*}
		This implies that $R_0\in\RR$ since by assumption the identity matrix is an optimal rotation.
		
		Uniqueness of $R_0$ is a straightforward consequence of \eqref{eq:estuniq}. Uniqueness (up to an infinitesimal rigid motion) of $u_0$ follows by arguing as in \cite[Theorem~5.1]{MaorMora}, recalling that displacements have zero average in our setting.
	\end{proof}
	
	The following proposition will be useful in both the liminf and the limsup inequalities to characterize the asymptotic behavior of the rescaled pressure potential. Note that, besides the presence of $\hat\pi$ in place of $\pi$, the integral at the left-hand side of \eqref{eq:neweq8} differs from the rescaled pressure potential in the total energy whenever $R_\eps$ is not an optimal rotation.
	
	\begin{prop}\label{prop:new}
		Let $\hat\pi$ be a function as in Lemma~\ref{lemma:hatpi1} and let $y_\eps\in \mathring{W}^{1,p}(\Omega;\R^n)$ satisfy \eqref{eq:estdet}. Assume there exist $R_\eps\in SO(n)$ converging to $R_0\in\RR$ such that the corresponding displacements $u_\eps$, defined as in \eqref{eq:displacement}, weakly converge in $\mathring{W}^{1,p}(\Omega;\R^n)$ to $u_0\in\mathring{H}^{1}(\Omega;\R^n)$. Then,
		\begin{align}
			\liminf_{\eps\to 0} \frac1\eps \int_\Omega \left(\hat\pi(y_\eps)\det\nabla y_\eps-\hat\pi(y_{R_\eps})\right)\d x 
			& \geq
			\lim_{\eps\to 0} \frac1\eps \int_{\{|\det\nabla y_\eps-1|\le 1\}} \left(\hat\pi(y_\eps)\det\nabla y_\eps-\hat\pi(y_{R_\eps})\right)\d x \nonumber \\
			& = \int_{\partial\Omega}\pi(R_0x)n_{\partial\Omega}(x)\cdot u_0(x)\d\mc H^{n-1}(x). \label{eq:neweq8}
		\end{align}
	\end{prop}
	
	\begin{proof}
		We write
		\begin{align}\label{eq:I}
			\frac1\eps \int_\Omega \left(\hat\pi(y_\eps)\det\nabla y_\eps-\hat\pi(y_{R_\eps})\right)\d x
			& =	\frac 1\eps\int_\Omega\hat\pi(y_\eps)(\det\nabla y_\eps-1)\d x+\frac 1\eps\int_\Omega(\hat\pi(y_\eps)-\hat\pi(y_{R_\eps}))\d x \nonumber \\
			& =:I_\eps+II_\eps.
		\end{align}
		
		We start by considering $I_\eps$. 
		Let $\Omega^-_\eps$ and $G_\eps$ be defined as in \eqref{eq:omegapiumeno} and \eqref{eq:goodset}. 
		Since $(u_\eps)$ is bounded in $W^{1,p}(\Omega;\R^n)$, property \eqref{eq:goodsetvanish} still holds.
		Moreover, by \eqref{eq:sviluppodet} 
		\begin{equation*}
			\det\nabla y_\eps(x)=\det(I+\eps\nabla u_\eps(x))=1+\sum_{k=1}^n\eps^k\iota_k(\nabla u_\eps(x))\qquad\text{for a.e.\ }x\in \Omega.
		\end{equation*} 
		Since by \eqref{eq:estik} we have that for $k=1,\dots,n$ 
		\begin{equation*}
			|\eps^k\iota_k(\nabla u_\eps(x))|\le C\eps^k|\nabla u_\eps(x)|^k\le C\eps^\frac k2\le C\eps^\frac 12\qquad\text{for a.e.\ }x\in G_\eps,
		\end{equation*}
		we deduce that $G_\eps\subset \Omega^-_\eps$ for $\eps$ small enough. Therefore, using the nonnegativity of $\hat\pi$ and \eqref{eq:sviluppodet} again, we obtain
		\begin{align*}
			I_\eps & \ge \frac 1\eps\int_{\Omega_\eps^-}\hat\pi(y_\eps)(\det\nabla y_\eps-1)\d x \\
			& = \int_{G_\eps}\hat\pi (y_\eps)\div u_\eps\d x+\sum_{k=2}^{n}\eps^{k-1}\int_{G_\eps}\hat\pi(y_\eps)\iota_k(\nabla u_\eps)\d x
			+\frac1\eps\int_{\Omega^-_\eps\setminus G_\eps}\hat\pi(y_\eps)(\det\nabla y_\eps-1)\d x \\
			&=:J^1_\eps+J^2_\eps+J^3_\eps.
		\end{align*}
		We first show that 
		\begin{equation}\label{eq:J1eps}
			\lim\limits_{\eps\to 0}J_\eps^1=\int_{\Omega}\pi(R_0x)\div u_0(x)\d x.
		\end{equation}
		Indeed, since $\hat\pi=\pi$ on $\mathcal O$, we may write
		\begin{align}
			\left| J^1_\eps-\int_{\Omega}\pi(R_0x)\div u_0(x)\d x\right|\le & \int_{G_\eps}|\hat\pi(R_\eps x+\eps R_\eps u_\eps(x))-\hat\pi(R_0 x)||\div u_\eps|\d x \nonumber \\
			&+\left|\int_{G_\eps}\pi(R_0 x)\div u_\eps\d x-\int_{\Omega}\pi(R_0 x)\div u_0\d x\right|. \label{eq:j1epart}
		\end{align}
		Using the Lipschitz continuity of $\hat{\pi}$ and the definition of $G_\eps$, the first integral at the right-hand side can be bounded as follows:
		\begin{align*}
			\int_{G_\eps}|\hat\pi(R_\eps x+\eps R_\eps u_\eps(x))-\hat\pi(R_0 x)||\div u_\eps|\d x
			&\le C|R_\eps-R_0|\|\nabla u_\eps\|_p+C\eps\int_{G_\eps}|u_\eps||\nabla u_\eps|\d x\\
			&\le C|R_\eps-R_0|+C\eps^\frac 12 \|u_\eps\|_p\\
			&\le C|R_\eps-R_0|+C\eps^\frac 12,
		\end{align*}
		where the last term goes to zero, as $\eps\to 0$. 
		Since $\chi_{G_\eps}$ converges to $1$ boundedly in measure, we have that $\chi_{G_\eps}\div u_\eps\wto \div u_0$ weakly in $L^p(\Omega)$, 
		hence the second term in \eqref{eq:j1epart} goes to zero, as well. This proves \eqref{eq:J1eps}.
		
		We now prove that both $J^2_\eps$ and $J^3_\eps$ converge to $0$, as $\eps\to 0$. By \eqref{eq:estik} and \eqref{eq:displacement}, since $\hat\pi$ is bounded, we obtain
		\begin{equation*}
			|J^2_\eps| \le C\sum_{k=2}^n\eps^{k-1}\int_{G_\eps} |\nabla u_\eps|^k\d x\ 
			=C \sum_{k=2}^n \eps^{k-1}\int_{G_\eps}|\nabla u_\eps|^{k-p}|\nabla u_\eps|^p\d x.
		\end{equation*}
		Since $|\nabla u_\eps|\le \eps^{-1/2}$ on $G_\eps$, we have that
		\begin{align*}
			|J^2_\eps|&\le C\sum_{k=2}^n\eps^\frac{k+p-2}{2}\|\nabla u_\eps\|_p^p\le C\eps^\frac p2\to 0.
		\end{align*}
		To bound $J^3_\eps$ we use \eqref{eq:estdet} and deduce
		\begin{equation*}
			|J^3_\eps| \le \frac{C}\eps {\|\det\nabla y_\eps-1\|_{L^2(\Omega_\eps^-)}}|\Omega\setminus G_\eps|^\frac 12
			 \le C |\Omega\setminus G_\eps|^\frac 12,
		\end{equation*}
		which vanishes by \eqref{eq:goodsetvanish}.
		
		By combining the previous inequalities we conclude that
		\begin{equation}\label{eq:liminf2}
			\liminf\limits_{\eps\to 0}I_\eps\ge \lim\limits_{\eps\to 0} \frac 1\eps\int_{\Omega_\eps^-}\hat\pi(y_\eps)(\det\nabla y_\eps-1)\d x = \int_{\Omega}\pi(R_0x)\div u_0(x)\d x.
		\end{equation}
		We now claim that 
		\begin{equation}\label{eq:liminf3}
			\lim\limits_{\eps\to 0}II_\eps= \lim\limits_{\eps\to 0} \frac 1\eps\int_{\Omega_\eps^-}(\hat\pi(y_\eps)-\hat\pi(y_{R_\eps}))\d x=\int_{\Omega}\nabla\pi(R_0 x)\cdot R_0u_0(x)\d x.
		\end{equation}
		Assuming this is true, the thesis follows by \eqref{eq:I}, \eqref{eq:liminf2}, \eqref{eq:liminf3}, and the Divergence Theorem,
		since
		\begin{equation*}
			\pi(R_0x)\div u_0(x)+\nabla\pi(R_0 x)\cdot R_0u_0(x)=\div (\pi(R_0 x)u_0(x)).
		\end{equation*}
		
		To conclude we only need to prove \eqref{eq:liminf3}. We can write the integrand in $II_\eps$ as
		\begin{equation}\label{eq:integrand}
			\frac 1\eps(\hat\pi(y_\eps)-\hat\pi(y_{R_\eps}))=\frac1\eps (\hat\pi(R_\eps x+\eps R_\eps u_\eps(x))-\hat\pi(R_\eps x))
		\end{equation}
		and owing to the Lipschitz continuity of $\hat{\pi}$ we have
		\begin{equation}\label{eq:boundLip}
			\frac 1\eps |\hat\pi(y_\eps)-\hat\pi(y_{R_\eps})| \leq C|u_\eps(x)| \qquad \text{for a.e.\ } x\in\Omega.
		\end{equation}
		Since $(u_\eps)$ is bounded in $L^p(\Omega;\R^n)$ and $|\Omega\setminus\Omega_\eps^-|\leq |\Omega\setminus G_\eps|\to0$ by \eqref{eq:goodsetvanish},
		we deduce that
		\begin{equation*}
			\lim\limits_{\eps\to 0} \frac 1\eps\int_{\Omega\setminus\Omega_\eps^-}(\hat\pi(y_\eps)-\hat\pi(y_{R_\eps}))\d x=0.
		\end{equation*}
		Hence, proving \eqref{eq:liminf3} is equivalent to show that
		\begin{equation}\label{eq:liminf3bis}
			\lim\limits_{\eps\to 0}II_\eps=\int_{\Omega}\nabla\pi(R_0 x)\cdot R_0u_0(x)\d x.
		\end{equation}
		On the other hand, $u_\eps\to u_0$ strongly in $L^1(\Omega;\R^n)$ by compact embedding. Thus, by \eqref{eq:boundLip} and the Generalised Dominated Convergence Theorem,  \eqref{eq:liminf3bis} is proved if we show that the integrand \eqref{eq:integrand} converges a.e.\ to $\nabla\pi(R_0 x)\cdot R_0u_0(x)$.
		
		To this aim, we first note that, up to subsequences,
		\begin{equation*}
			\lim\limits_{\eps\to 0}\nabla\hat\pi (R_\eps x)=\nabla\hat\pi (R_0 x)\quad\text{for a.e.\ }x\in \Omega.
		\end{equation*}
		Indeed, the convergence is actually in $L^1(\mc O;\R^n)$. This can be easily proved by approximating $\nabla\hat\pi$ with functions in
		$C^0(\overline{\mc O};\R^n)$. 
		Now, by Rademacher Theorem (we point out that we are working with a countable sequence of rotations $R_\eps$) for almost every $x\in \Omega$ we have
		\begin{equation*}
			\frac{\hat\pi(R_\eps x+\eps R_\eps u_\eps(x))-\hat\pi(R_\eps x)}{\eps}=\nabla\hat\pi (R_\eps x)\cdot R_\eps u_\eps(x)+\frac1\eps o(\eps|u_\eps(x)|).
		\end{equation*}
		Since $u_\eps\to u_0$ a.e., up to subsequences, and $\hat\pi\equiv\pi$ on $\mc O$, we deduce the desired convergence.
		This concludes the proof.
	\end{proof}
	
	With the result of Proposition~\ref{prop:new} at hand, we are now in a position to state and prove the liminf and the limsup inequalities for the energy functionals $\frac 1{\eps^2}\mc E_\eps$. 
	
	\begin{prop}[\textbf{Liminf inequality}]\label{prop:gammaliminf}
		Assume \ref{hyp:W1}--\ref{hyp:W5}, \ref{hyp:pi1}, and \ref{hyp:pi2}. For every $\eps\in(0,1)$ let $y_\eps\in \mathring{W}^{1,p}(\Omega;\R^n)$ be such that there exist $R_\eps\in SO(n)$ converging to $R_0\in\RR$ and the corresponding displacements $u_\eps$, defined as in \eqref{eq:displacement}, weakly converge in $\mathring{W}^{1,p}(\Omega;\R^n)$ to $u_0\in\mathring{H}^{1}(\Omega;\R^n)$. Then
		\begin{equation*}
			\EE_0(u_0,R_0)\le \liminf\limits_{\eps\to 0}\frac{1}{\eps^2}\EE_\eps(y_\eps),
		\end{equation*}
		where $\EE_0\colon \mathring{H}^{1}(\Omega;\R^n)\times\RR\to \R$ is defined by
		\begin{equation}\label{eq:E0}
			\EE_0(u_0,R_0):=
			\frac 12\int_\Omega Q(x, e(u_0)(x))\d x+\int_{\partial\Omega}\pi(R_0x)n_{\partial\Omega}(x)\cdot u_0(x)\d\mc H^{n-1}(x).
		\end{equation}
		The density $Q(x,\cdot)$ is the quadratic form given by 
		\begin{equation*}
			Q(x,F)=D^2_F W(x,I)F:F \qquad \text{for } F\in\R^{n\times n},
		\end{equation*}
		and $e(u_0)$ denotes the symmetric gradient of $u_0$.
	\end{prop}
	
	\begin{rmk}\label{rmk:invariant}
		For any optimal rotation $R_0\in\RR$ the functional $\EE_0(\cdot, R_0)$ is invariant under perturbations by infinitesimal rigid motions. Indeed, if $u_0'(x)=u_0(x)+Ax$ with $A\in \R^{n\times n}_{\rm skew}$, 
		then clearly $e(u_0')=e(u_0)$ and by \eqref{eq:EL2}
		\begin{equation*}
			\int_{\partial\Omega}\pi(R_0x)n_{\partial\Omega}(x)\cdot u'_0(x)\d\mc H^{n-1}(x)=
			\int_{\partial\Omega}\pi(R_0x)n_{\partial\Omega}(x)\cdot u_0(x)\d\mc H^{n-1}(x).
		\end{equation*}
	\end{rmk}
	
	\begin{proof}[Proof of Proposition~\ref{prop:gammaliminf}]
		Without loss of generality we can assume 
		\begin{equation*}
			\liminf\limits_{\eps\to 0}\frac{1}{\eps^2}\EE_\eps(y_\eps)<+\infty,
		\end{equation*}	
		so that $y_\eps\in Y^p$ and, up to subsequence, $\EE_\eps(y_\eps)\leq C\eps^2$. 
		By Lemma~\ref{lemma:est1} and Proposition~\ref{prop:compactness} there exist a (possibly different) sequence $(R_\eps')\subset SO(n)$ such that, up to subsequences, $R'_\eps\to R_0$, the corresponding displacements $u_\eps'$ satisfy \eqref{eq:estgp} and, up to subsequences, weakly converge to $u_0+Ax$ for some $A\in \R^{n\times n}_{\rm skew}$. 
		However, by Remark~\ref{rmk:invariant} we can assume, without loss of generality, that $R_\eps=R_\eps'$ and so, $u_\eps=u_\eps'$ and $A=0$.

		Let $\widehat{\mc E}_\eps$ be the auxiliary energy defined as in \eqref{eq:energy1} with $\pi$ replaced by the function $\hat\pi$ given by Lemma~\ref{lemma:hatpi1}. By the properties of $\hat\pi$ we have
		\begin{align}
			\frac{1}{\eps^2}\EE_\eps(y_\eps)\ge \frac{1}{\eps^2}\widehat\EE_\eps(y_\eps) = & \ \frac{1}{\eps^2}\int_\Omega W(x,\nabla y_\eps)\d x+\frac 1\eps\int_{\Omega}		(\hat\pi(y_\eps)\det\nabla y_\eps-\hat\pi(y_{R_\eps}))\d x \nonumber \\
			& + \frac 1\eps\int_{\Omega} (\hat\pi(y_{R_\eps})-\hat\pi(x))\d x. \label{eq:terms}
		\end{align}	
		Arguing as in \cite[Proof of Theorem~2.4]{AgoDMDS} one can prove that
		\begin{equation*}
			\liminf\limits_{\eps\to 0}\frac{1}{\eps^2}\int_\Omega W(x,\nabla y_\eps)\d x\ge 	\frac 12\int_\Omega Q(x,e(u_0)(x))\d x.
		\end{equation*}
		Since condition \eqref{eq:estdet} is satisfied by Lemma~\ref{lemma:est1}, we can apply Proposition~\ref{prop:new} and we obtain
		\begin{equation*}
			\liminf_{\eps\to 0} \frac1\eps \int_\Omega \left(\hat\pi(y_\eps)\det\nabla y_\eps-\hat\pi(y_{R_\eps})\right)\d x 
			\geq \int_{\partial\Omega}\pi(R_0x)n_{\partial\Omega}(x)\cdot u_0(x)\d\mc H^{n-1}(x).
		\end{equation*}
		Finally, assumption \eqref{eq:IinR} guarantees that the last term in \eqref{eq:terms} is nonnegative. This proves the desired inequality.
	\end{proof}
	
	\begin{prop}[\textbf{Limsup inequality}]\label{prop:gammalimsup}
		Assume \ref{hyp:W1}--\ref{hyp:W5}, \ref{hyp:pi1}, and \ref{hyp:pi2}.
		For every $(u_0,R_0)\in \mathring{H}^{1}(\Omega;\R^n)\times\RR$ there exist $(u_\eps,R_\eps)\in \mathring{W}^{1,p}(\Omega;\R^n)\times SO(n)$ such that $u_\eps\wto u_0$ weakly in $\mathring{W}^{1,p}(\Omega;\R^n)$, $R_\eps\to R_0$ and, setting $y_\eps(x):=R_\eps(x+\eps u_\eps(x))$, there holds
		\begin{equation*}	
			\limsup\limits_{\eps\to 0}\frac{1}{\eps^2}\EE_\eps(y_\eps)\le \EE_0(u_0,R_0),
		\end{equation*}
		where $\EE_0$ is the functional defined in \eqref{eq:E0}.
	\end{prop}
	
	\begin{proof}
		Let $(u_0,R_0)\in \mathring{H}^{1}(\Omega;\R^n)\times\RR$.
		By mollification there exists $(u_\eps)\subset\mathring{W}^{1,\infty}(\Omega;\R^n)$ such that
		\begin{equation}\label{eq:strong}
			u_\eps\to u_0\quad\text{strongly in }H^1(\Omega;\R^n)\qquad\text{ and }\qquad\eps^\frac 12\|u_\eps\|_{W^{1,\infty}}\le 1.
		\end{equation}
		We define $R_\eps:=R_0$, so that $y_\eps(x)=R_0(x+\eps u_\eps(x))$.
		
		We first observe that $y_\eps\in Y^p$ for $\eps$ small enough. Indeed, by \eqref{eq:baricenter} it has zero-average and by \eqref{eq:sviluppodet} it satisfies
		\begin{equation}\label{eq:detyyy}
			\det\nabla y_\eps(x)=\det(I+\eps\nabla u_\eps(x))=1+\sum_{k=1}^n\eps^k\iota_k(\nabla u_\eps(x))\qquad\text{for a.e.\ }x\in \Omega.
		\end{equation} 
		Since by \eqref{eq:estik} we have that for $k=1,\dots,n$ 
		\begin{equation*}
			|\eps^k\iota_k(\nabla u_\eps(x))|\le C\eps^k|\nabla u_\eps(x)|^k\le C\eps^\frac k2\le C\eps^\frac 12\qquad\text{ for a.e.\ }x\in \Omega,
		\end{equation*}
		for $\eps$ small enough we obtain
		\begin{equation}\label{eq:small1}
			|\det\nabla y_\eps(x)-1|\leq C\eps^\frac 12\le \frac 12\qquad\text{ for a.e.\ }x\in \Omega,
		\end{equation}
		and thus, $\det\nabla y_\eps>0$ a.e.\ in $\Omega$.
		
		By \eqref{eq:strong} we have that for $\eps$ small enough the set $y_\eps(\Omega)$ is contained in the neighborhood of $\mc O$ where $\pi$ and $\hat\pi$ coincide.
		Therefore, using also that $R_0\in \RR$, we can write
		\begin{equation*}
			\frac{1}{\eps^2}\EE_\eps(y_\eps)=\frac{1}{\eps^2}\int_{\Omega}W(x,I{+}\eps\nabla u_\eps(x))\d x+\frac 1\eps\int_{\Omega}\left(\hat \pi(y_\eps)\det \nabla y_\eps{-}\hat\pi(y_{R_0})\right)\d x.
		\end{equation*} 
		Arguing as in \cite[Proof of Theorem~2.4]{AgoDMDS}, one can show that
		\begin{equation*}
			\limsup\limits_{\eps\to 0}\frac{1}{\eps^2}\int_{\Omega}W(x,I+\eps\nabla u_\eps(x))\d x\le \frac 12\int_\Omega Q(x, e(u_0)(x))\d x.
		\end{equation*}
		On the other hand, by \eqref{eq:detyyy} we have that
		\begin{equation*}
			|\det\nabla y_\eps(x)-1|\leq \eps|\nabla u_\eps(x)| +C\eps \qquad\text{for a.e.\ }x\in \Omega,
		\end{equation*} 
		hence condition \eqref{eq:estdet} is satisfied and we can apply Proposition~\ref{prop:new}. By \eqref{eq:small1} we deduce 
		\begin{equation*}
			\lim\limits_{\eps\to 0}\frac 1\eps\int_{\Omega}\left(\hat \pi(y_\eps)\det \nabla y_\eps){-}\hat\pi(y_{R_0})\right)\d x=\int_{\partial\Omega}\pi(R_0x)n_{\partial\Omega}(x)\cdot u_0(x)\d\mc H^{n-1}(x).
		\end{equation*}
		This concludes the proof.
	\end{proof}
	
	\begin{rmk}\label{rmk:injectiveinf}
		If we include a.e.\ injectivity in the definition of the space $Y^p$ of admissible deformations (see Remark~\ref{rmk:injective}), the limsup inequality can be proved by means of the same recovery sequence. Indeed, by \cite[Theorem~5.5-1(b)]{Ciarlet} the deformations $y_\eps$ are a.e.\ injective owing to \eqref{eq:strong}.
	\end{rmk}
	
	Combining together the previous propositions, we can prove the main result of this section. It ensures that almost minimizers of the nonlinear energy strongly converge to minimizers of the limiting energy.
	
	\begin{thm}[\textbf{Convergence of almost minimizers}]\label{thm:convmin}
		Assume \ref{hyp:W1}--\ref{hyp:W5}, \ref{hyp:pi1}, and \ref{hyp:pi2}.
		If $(y_\eps)$ is a sequence of almost minimizers for the energies $\EE_\eps$, that is, 
		\begin{equation}\label{eq:almmin}
			\EE_\eps(y_\eps)\le \inf\limits_{\mathring{W}^{1,p}(\Omega;\R^n)}\EE_\eps+o(\eps^2), 
		\end{equation}
		then there exist $R_\eps\in SO(n)$ such that, up to passing to a subsequence, we have
		\begin{itemize}
			\item $u_\eps\to u_0$ strongly in $\mathring{W}^{1,p}(\Omega;\R^n)$ with $u_0\in \mathring{H}^{1}(\Omega;\R^n)$,\smallskip
			\item $R_\eps\to R_0$ with $R_0\in \RR$,
		\end{itemize}
		as $\eps\to0$.	
		Furthermore, the pair $(u_0,R_0)$ is a minimizer of $\EE_0$ on $\mathring{H}^{1}(\Omega;\R^n)\times \RR$ and
		\begin{equation}\label{eq:minEE_0}
			\lim\limits_{\eps\to 0} \frac{1}{\eps^2} \Big(\inf\limits_{\mathring{W}^{1,p}(\Omega;\R^n)}\EE_\eps\Big) =\min
			\big\{ \EE_0(u,R) : \ (u,R)\in\mathring{H}^{1}(\Omega;\R^n)\times \RR \big\}.
		\end{equation}
	\end{thm}
	
	\begin{proof}
		Let $(y_\eps)$ be a sequence of almost minimizers. By Corollary~\ref{cor:inf} we have that 
		\begin{equation*}
			\inf\limits_{\mathring{W}^{1,p}(\Omega;\R^n)}\EE_\eps \le 0, 
		\end{equation*}
		hence by Proposition~\ref{prop:compactness} there exist $u_0\in \mathring{H}^{1}(\Omega;\R^n)$ and $R_0\in \RR$ such that, up to a subsequence,
		\begin{equation*}
			u_\eps\wto u_0 \quad\text{weakly in }\mathring{W}^{1,p}(\Omega;\R^n)\qquad\text{ and }\qquad R_\eps\to R_0.
		\end{equation*}
		We now show that $(u_0,R_0)$ is a minimizer of $\EE_0$. To this aim let $(v,S)\in \mathring{H}^{1}(\Omega;\R^n)\times \RR$ and let $(v_\eps,S_\eps)$ be a recovery sequence for $(v,S)$, as in Proposition~\ref{prop:gammalimsup}. Let $z_\eps(x):=S_\eps(x+\eps v_\eps(x))$. By Proposition~\ref{prop:gammaliminf} we have
		\begin{align}\label{eq:vS}
			\EE_0(u_0,R_0)&\le \liminf\limits_{\eps\to 0}\frac{1}{\eps^2}\EE_\eps(y_\eps)\le
			\liminf\limits_{\eps\to 0}\Big(\inf\limits_{\mathring{W}^{1,p}(\Omega;\R^n)}\frac{1}{\eps^2}\EE_\eps\Big) \le
			\limsup\limits_{\eps\to 0}\Big(\inf\limits_{\mathring{W}^{1,p}(\Omega;\R^n)}\frac{1}{\eps^2}\EE_\eps\Big) \nonumber \\
			& \le \limsup\limits_{\eps\to 0}\frac{1}{\eps^2}\EE_\eps(z_\eps)\le\EE_0(v,S).
		\end{align}
		This implies that $\EE_0$ is minimized at $(u_0,R_0)$ and, as a consequence, \eqref{eq:minEE_0} hold.
		
		To conclude, it remains to prove that $u_\eps\to u_0$ strongly in $W^{1,p}(\Omega;\R^n)$.
		We adapt the argument in \cite[Theorem~2.5]{AgoDMDS} to our framework. We claim that the following properties hold:
		\begin{itemize}
			\item[(a)] $\chi_{G_\eps}e(u_\eps)\to e(u_0)$ strongly in $L^2(\Omega;\R^{n\times n}_{\rm sym})$, where the set $G_\eps$ is defined as in \eqref{eq:goodset};
			\item[(b)] the sequence $\Big(\frac{1}{\eps^p}\dist^p(\nabla y_\eps;SO(n))\Big)$ is equi-integrable;
			\item[(c)] the sequence $(|\nabla u_\eps|^p)$ is equi-integrable.
		\end{itemize}
		The thesis follows from (a) and (b), by using Vitali's convergence theorem together with Korn's second inequality, see \cite[proof of Theorem~2.5]{AgoDMDS} for more details.
		
		We now prove (a). By choosing $(v,S)=(u_0,R_0)$ in \eqref{eq:vS} we deduce
		\begin{equation*}
			\lim\limits_{\eps\to 0}\frac{1}{\eps^2}\EE_\eps(y_\eps)=\EE_0(u_0,R_0).
		\end{equation*}
		By \eqref{eq:terms} and assumption \eqref{eq:IinR} we have
		\begin{equation*}
			\frac{1}{\eps^2}\EE_\eps(y_\eps)\ge \frac{1}{\eps^2}\int_{\Omega} W(x,\nabla y_\eps)\d x+\frac1\eps \int_\Omega \left(\hat\pi(y_\eps)\det\nabla y_\eps-\hat\pi(y_{R_\eps})\right)\d x.
		\end{equation*}
		Therefore, letting $\eps \to 0$ and applying Proposition~\ref{prop:new} yield
		\begin{align*}
			\limsup\limits_{\eps\to 0}\frac{1}{\eps^2}\int_{\Omega} W(x,\nabla y_\eps)\d x&\le \EE_0(u_0,R_0)-\int_{\partial\Omega}\pi(R_0x)n_{\partial\Omega}(x)\cdot u_0(x)\d\mc H^{n-1}(x)\\
			&=\frac 12\int_\Omega Q(x, e(u_0))\d x.
		\end{align*}
		On the other hand, by Taylor expansion of $W$ around $I$ and by the weak convergence of $\chi_{G_\eps}e(u_\eps)$ to $e(u_0)$ in $L^2(\Omega;\R^{n\times n}_{\rm sym})$ (see
		property (ii) in the proof of Proposition~\ref{prop:compactness}) we obtain
		\begin{align*}
			\limsup\limits_{\eps\to 0}\frac{1}{\eps^2}\int_{\Omega} W(x,\nabla y_\eps)\d x&\ge 	\limsup\limits_{\eps\to 0}\frac 12\int_\Omega Q(x,\chi_{G_\eps}e(u_\eps))\d x\\
			&\ge \liminf\limits_{\eps\to 0}\frac 12\int_\Omega Q(x,\chi_{G_\eps}e(u_\eps))\d x\\
			&\ge \frac 12\int_\Omega Q(x,e(u_0))\d x,
		\end{align*}
		see, e.g., \cite{AgoDMDS,DMNegPerc,MaorMora}.
		Combining the previous inequalities yields
		\begin{equation}\label{eq:limit}
			\lim\limits_{\eps\to 0}\frac 12\int_\Omega Q(x,\chi_{G_\eps}e(u_\eps))\d x=\frac 12\int_\Omega Q(x,e(u_0))\d x.
		\end{equation}
		Since $\chi_{G_\eps}e(u_\eps)\wto e(u_0)$ weakly in $L^2(\Omega;\R^{n\times n}_{\rm sym})$
		and the quadratic form $Q(x,\cdot)$ is coercive on $\R^{n\times n}_{\rm sym}$ by \ref{hyp:W2}, \ref{hyp:W4}, and \ref{hyp:W5}, equation \eqref{eq:limit} proves claim (a).
		
		To show claim (b) one can repeat verbatim the proof in \cite[Theorem~2.5]{AgoDMDS}.
		
		We now prove claim (c). Given $\alpha>p$ and $\eta>0$, we deduce by (b) that there exists $M_\eta>0$ such that, setting 
		\begin{equation*}
			f_1^{\eps,\eta}(x):=\dist(\nabla y_\eps(x);SO(n))\chi_{E^{\eps,\eta}}(x),\qquad f_2^{\eps,\eta}(x):=\dist(\nabla y_\eps(x);SO(n))\chi_{\Omega\setminus E^{\eps,\eta}}(x),
		\end{equation*}
		where
		\begin{equation*}
			E^{\eps,\eta}:=\left\{x\in\Omega: \displaystyle \frac{1}{\eps^p}\dist^p(\nabla y_\eps(x);SO(n))\ge M_\eta\right\},
		\end{equation*}
		we have that
		\begin{equation}\label{eq:serve}
			\left\|\frac{f_1^{\eps,\eta}}{\eps}\right\|_p^p\le \eta\qquad\text{ and }\qquad \left\|\frac{f_2^{\eps,\eta}}{\eps}\right\|_\alpha^\alpha\le |\Omega|M_\eta^\frac\alpha p.
		\end{equation}
		Theorem~\ref{lemma:rigidity2} now ensures the existence of $\widetilde{R}_{\eps,\eta}\in SO(n)$ and of $g_1^{\eps,\eta}, g_2^{\eps,\eta}$ such that
		\begin{equation}\label{eq:serve2}
			\nabla y_\eps=\widetilde{R}_{\eps,\eta}+g_1^{\eps,\eta}+g_2^{\eps,\eta}\qquad\text{ and }\qquad \|g_1^{\eps,\eta}\|_p\le C\|f_1^{\eps,\eta}\|_p,\ \|g_2^{\eps,\eta}\|_\alpha\le C\|f_2^{\eps,\eta}\|_\alpha.
		\end{equation} 
		Since $\nabla y_\eps=R_\eps+\eps R_\eps\nabla u_\eps$, we deduce that
		\begin{equation}\label{eq:equality}
			\frac{\widetilde{R}_{\eps,\eta}-R_\eps}{\eps}=R_\eps\nabla u_\eps-\frac{g_1^{\eps,\eta}}{\eps}-\frac{g_2^{\eps,\eta}}{\eps},
		\end{equation}
		hence
		\begin{equation}\label{eq:estimateeta}
			\frac{|\widetilde{R}_{\eps,\eta}-R_\eps|^p}{\eps^p}\le C\left(\|\nabla u_\eps\|_p^p+\left\|\frac{g_1^{\eps,\eta}}{\eps}\right\|_p^p+\left\|\frac{g_2^{\eps,\eta}}{\eps}\right\|_p^p\right)\le C(1+\eta+M_\eta),
		\end{equation}
		where the last inequality follows from H\"older's inequality, \eqref{eq:serve}, and \eqref{eq:serve2}.
		
		On the other hand, by \eqref{eq:equality} we can write
		\begin{equation*}
			\nabla u_\eps=R_\eps^T\left(	\frac{\widetilde{R}_{\eps,\eta}-R_\eps}{\eps}+\frac{g_1^{\eps,\eta}}{\eps}+\frac{g_2^{\eps,\eta}}{\eps}\right).
		\end{equation*}
		Thus, by \eqref{eq:serve}, \eqref{eq:serve2}, and \eqref{eq:estimateeta} we have that
		for every measurable set $A\subset\Omega$ 
		\begin{align*}
			\int_A|\nabla u_\eps|^p\d x&\le C\left(\frac{|\widetilde{R}_{\eps,\eta}-R_\eps|^p}{\eps^p}|A|+\int_{A}\left|\frac{g_1^{\eps,\eta}}{\eps}\right|^p\d x +\int_{A}\left|\frac{g_2^{\eps,\eta}}{\eps}\right|^p\d x\right)\\
			&\le C\left((1+\eta+M_\eta)|A|+\eta+\left\|\frac{f_2^{\eps,\eta}}{\eps}\right\|_\alpha^p|A|^{1-\frac p\alpha}\right)\\
			&\le C \left((1+\eta+M_\eta)|A|+\eta+M_\eta|A|^{1-\frac p\alpha}\right).
		\end{align*}
		Now, for every $\delta>0$ we can choose first $\eta=\eta(\delta)$ and then $\omega=\omega(\delta,\eta)$ in such a way that the right-hand side above is less than $\delta$
		for every measurable set $A\subset\Omega$ with $|A|<\omega$. This proves claim (c) and concludes the proof of the theorem.
	\end{proof}

	\section{Pressure loads of arbitrary sign}\label{sec:general}
	
	Here we extend the results of the previous section to pressure loads whose intensity $\pi$ is not necessarily nonnegative (and still satisfies \ref{hyp:pi1}). To deal with the negative part of $\pi$ we need to assume an additional bound from below for $W(\cdot,F)$ in terms of $\det F$:\smallskip
	\begin{enumerate}[label=\textup{(W6)}]
		\item\label{hyp:W6} $W(x,F)\ge c_2 g_q(|\det F-1|)$ for a.e.\ $x\in\Omega$ and for every $F\in \R^{n\times n}$, for some $q\in [1,2]$,\smallskip
	\end{enumerate}
	where $g_q$ is defined as in \eqref{eq:gp} and $c_2>0$ is a constant independent of $x$.
	According to the value of $q$ in \ref{hyp:W6}, we assume $\pi$ to satisfy the following condition:\smallskip
	\begin{enumerate}[label=\textup{($\pi$3)}]
		\item\label{hyp:pi3} if $q=1$, $\pi^-$ is bounded;
		if $q\in(1,2]$, $\pi^-(y)\le C(1+|y|^\frac{p}{q'})$ for every $y\in \R^n$.\smallskip
	\end{enumerate}
	We note that the growth condition in \ref{hyp:pi3} is at most linear, since $p,q\in (1,2]$ implies $p/q'\in(0,1]$.
	
	In the current framework the energy $\EE_\eps$ is defined as in \eqref{eq:energy1} with the set of admissible deformations $Y^p$ replaced by
	\begin{equation}\label{eq:Xpq}
		Y^p_q:=\left\{y\in Y^p: \ \det\nabla y\in L^q(\Omega) \right\}.
	\end{equation}
	Owing to \ref{hyp:pi3} the energy is well defined on $Y^p_q$: indeed, if $y\in Y^p_q$, then the composition $\pi^-\circ y$ belongs to $L^{q'}(\Omega)$ and thus, $\pi(y)\det\nabla y$ is integrable. Clearly, all rigid motions $y_R$ with $R\in SO(n)$ are still admissible deformations. 
	As observed in Remark~\ref{rmk:injective}, also in this setting the a.e.\ injectivity condition can be included in the definition of $Y^p_q$ without altering the results of this section.
	
	As in the previous section we need a Lipschitz continuous function that extends $\pi$ outside a neighborhood of $\overline{\mc O}$, is below $\pi$ everywhere, and satisfies the same growth condition \ref{hyp:pi3} as $\pi$.

	\begin{lemma}\label{lemma:hatpi2}
		Assume conditions \ref{hyp:pi1} and \ref{hyp:pi3}. Then, there exists a Lipschitz continuous function $\hat\pi\colon \R^n\to \R$ such that $\hat\pi$ coincides with $\pi$ in a neighborhood of $\overline{\mc O}$, $\hat\pi(y)\le\pi(y)$ for all $y\in \R^n$, and $\hat\pi$ has the following property: if $q=1$, $\hat\pi$ is bounded; if $q\in(1,2]$, $|\hat\pi(y)|\le C(1+|y|^\frac{p}{q'})$ for every $y\in \R^n$.
	\end{lemma} 
	
	\begin{proof}
		We consider only the case $q\in(1,2]$, being the case $q=1$ analogous and even simpler. 
		Let $\overline{C}>0$ be a constant for which \ref{hyp:pi3} is satisfied and let 
		\begin{equation*}
			h(y):=\big(\overline{C}(1+|y|^\frac{p}{q'})\big)\vee(1+\overline{C}).
		\end{equation*}
		Since $p/q'\in (0,1]$, the function $h$ is Lipschitz continuous in the whole of $\R^n$ and $\pi\ge -h$ by \ref{hyp:pi3}. Hence we can apply 
		Lemma~\ref{lemma:hatpi1} to the function $\Pi:=\pi+h$. This provides us with a function $\widehat\Pi$. It is now easy to check that the function
		$\hat{\pi}:=\widehat\Pi-h$ has all the required properties.	
	\end{proof}
	
	Under this new set of assumptions, the results of Section~\ref{sec:nonnegative} can be modified as follows.
	
	\begin{lemma}\label{lemma:est2}
		Assume \ref{hyp:W1}--\ref{hyp:W6}, \ref{hyp:pi1}, and \ref{hyp:pi3}. Then there exists $\eps_0\in (0,1)$ such that, if $\EE_\eps(y_\eps)\le C\eps^2$ for every $\eps\in(0,\eps_0)$, then there holds
		\begin{equation}\label{eq:estdet2}
			\int_{\Omega}g_q(|\det\nabla y_\eps(x)-1|)\d x\le C\eps^2 \quad \text{for every }\eps\in(0,\eps_0). 
		\end{equation}
		Furthermore, there exist constant rotations $R_\eps\in SO(n)$ such that the rescaled displacements $u_\eps$ defined by \eqref{eq:displacement} satisfy
		\begin{equation}\label{eq:estgp2}
			\int_{\Omega}g_p(\eps|\nabla u_\eps(x)|)\d x\le C\eps^2 \quad \text{for every }\eps\in(0,\eps_0)
		\end{equation}
		and are uniformly bounded in $W^{1,p}(\Omega;\R^n)$.
		
		If, moreover, $(R'_\eps)\subset SO(n)$ is another sequence for which the rescaled displacements, defined as in \eqref{eq:displacement}, satisfy \eqref{eq:estgp2}, then
		\begin{equation*}
			|R_\eps-R'_\eps|\le C\eps
		\end{equation*} 
		for every $\eps\in (0,\eps_0)$.
		Finally, one has 
		\begin{equation*}
			-C\eps^2\le\inf\limits_{\mathring{W}^{1,p}(\Omega;\R^n)}\EE_\eps\le 0 \quad \text{for every }\eps\in(0,\eps_0).
		\end{equation*}
	\end{lemma}
	
	\begin{proof}
		The choice of $\eps_0$ will be made throughout the proof. We follow the lines of the proof of Lemma~\ref{lemma:est1}. 
		By Lemma~\ref{lemma:hatpi2} inequality \eqref{eq:Ehat} still holds. By \ref{hyp:W5} and Theorem~\ref{lemma:prigidity} there exists a sequence $(R_\eps)\subset SO(n)$ such that
		\begin{equation*}
			\int_\Omega g_p(|\nabla y_\eps(x)-R_\eps|)\d x\le C\int_\Omega W(x,\nabla y_\eps(x))\d x.
		\end{equation*}
		and
		\begin{equation}\label{eq:West2}
			\int_\Omega W(x,\nabla y_\eps)\d x\le C\eps^2+C\eps\|y_\eps-y_{R_\eps}\|_1+\eps\int_\Omega |\hat\pi (y_\eps)| |\det\nabla y_\eps-1|\d x.
		\end{equation}
		We denote by $P_\eps$ the last term in the above inequality. In the following, $c_2$ is the constant in condition \ref{hyp:W6}.
		If $q=1$, the function $\hat\pi$ is bounded and thus, recalling the definition \eqref{eq:omegapiumeno} of the sets $\Omega_\eps^\pm$, 
		we have
		\begin{align*}
			P_\eps & \le C\eps\|\det\nabla y_\eps-1\|_{L^2(\Omega_\eps^-)}+C\eps\|\det\nabla y_\eps-1\|_{L^1(\Omega_\eps^+)} \\
			&\le C\eps^2 +\frac{c_2}4 \|\det\nabla y_\eps-1\|^2_{L^2(\Omega_\eps^-)}+ C\eps\|\det\nabla y_\eps-1\|_{L^1(\Omega_\eps^+)} \\
			&\le C\eps^2  + \Big(\frac{c_2}2+C\eps\Big) \int_{\Omega}g_1(|\det\nabla y_\eps-1|)\d x,
		\end{align*}
		where we used Cauchy's inequality and \eqref{eq:propgp}.
		If instead $q\in(1,2]$, the function $\hat\pi$ is Lipschitz continuous and satisfies a $p/q'$-growth condition with $p/q'\leq1$, so that we obtain
		\begin{equation*}
			P_\eps \le C\eps \int_{\Omega_\eps^-}(1+|y_\eps-y_{R_\eps}|)|\det\nabla y_\eps-1|\d x+ C\eps\int_{\Omega_\eps^+}(1+|y_\eps-y_{R_\eps}|^\frac{p}{q'})|\det\nabla y_\eps-1|\d x.
			\end{equation*}
		Using that $|\det\nabla y_\eps-1|\leq 1$ on $\Omega_\eps^-$ and applying H\"older's inequality, from the previous equation we deduce
		\begin{align*}
			P_\eps &\le C\eps\|\det\nabla y_\eps-1\|_{L^2(\Omega_\eps^-)}+C\eps\|y_\eps-y_{R_\eps}\|_{W^{1,p}} +C\eps\|\det\nabla y_\eps-1\|_{L^q(\Omega_\eps^+)}\\
			&\qquad +C\eps\int_{\Omega_\eps^+}|y_\eps-y_{R_\eps}|^p\d x+C\eps\int_{\Omega_\eps^+}|\det\nabla y_\eps-1|^q\d x \\
			&\le C\eps^2 + \frac{c_2}4 \|\det\nabla y_\eps-1\|^2_{L^2(\Omega_\eps^-)} +C\eps^{q'} + \Big(\frac{c_2}4+C\eps\Big) \|\det\nabla y_\eps-1\|^q_{L^q(\Omega_\eps^+)} \\
			&\qquad +C\eps \|y_\eps-y_{R_\eps}\|_{W^{1,p}}+C\eps\|y_\eps-y_{R_\eps}\|_{W^{1,p}}^p \\
			&\le C\eps^2 + \Big(\frac{c_2}2+C\eps\Big) \int_{\Omega}g_q(|\det\nabla y_\eps-1|)\d x 
			+C\eps \|y_\eps-y_{R_\eps}\|_{W^{1,p}}+C\eps\|y_\eps-y_{R_\eps}\|_{W^{1,p}}^p,
		\end{align*}
		where we used Young's inequality, \eqref{eq:propgp}, and the fact that $q'\geq2$.
		Combining \eqref{eq:West2} with the previous bounds on $P_\eps$, we obtain that 
		\begin{equation}\label{eq:almostfinalest}
			\begin{aligned}
				\int_\Omega W(x,\nabla y_\eps)\d x & \le  C\eps^2+C\eps\|y_\eps-y_{R_\eps}\|_{W^{1,p}}+C\eps\|y_\eps-y_{R_\eps}\|_{W^{1,p}}^p  \\
				& \qquad + \Big(\frac{c_2}2+C\eps\Big) \int_{\Omega}g_q(|\det\nabla y_\eps-1|)\d x 
			\end{aligned}
		\end{equation}
		in both cases $q=1$ and $q\in (1,2]$. 
		
		Now, if $\eps_0\le c_2/(4\overline{C})$, where $\overline{C}$ is a constant for which \eqref{eq:almostfinalest} is true, then by \ref{hyp:W6} we deduce that
		\begin{equation}\label{claim}
			\int_{\Omega}g_q(|\det\nabla y_\eps-1|)\d x\le C\eps^2+C\eps\|y_\eps-y_{R_\eps}\|_{W^{1,p}}+C\eps\|y_\eps-y_{R_\eps}\|_{W^{1,p}}^p
		\end{equation}
		for every $\eps\in(0,\eps_0)$.
		Combining \eqref{claim} and \eqref{eq:almostfinalest} yields
		\begin{equation}\label{eq:estW2}
			\begin{aligned}
				\int_\Omega W(x,\nabla y_\eps)\d x&\le C\eps^2+C\eps\|y_\eps-y_{R_\eps}\|_{W^{1,p}}+C\eps\|y_\eps-y_{R_\eps}\|_{W^{1,p}}^p\\
				&= C\eps^2(1+\|u_\eps\|_{W^{1,p}}+\eps^{p-1}\|u_\eps\|_{W^{1,p}}^p).
			\end{aligned}
		\end{equation}
		Arguing as in \eqref{eq:L2bound}--\eqref{eq:Lpbound2} and using Poincar\'e-Wirtinger inequality we deduce that
		\begin{equation*}
			\| u_\eps\|_{W^{1,p}}^p\le C(1+\|u_\eps\|_{W^{1,p}}+\eps^{p-1}\|u_\eps\|_{W^{1,p}}^p).
		\end{equation*}
		Up to choosing $\eps_0$ smaller, if needed, we obtain
		\begin{equation*}
			\| u_\eps\|_{W^{1,p}}^p\le C(1+\|u_\eps\|_{W^{1,p}}),
		\end{equation*}
		which implies $\|u_\eps\|_{W^{1,p}}\le C$.
		
		Inequality \eqref{eq:estdet2} now follows easily from \eqref{claim}, while \eqref{eq:estgp2} is a consequence of \eqref{eq:estW2} and \ref{hyp:W5}. The last two statements of the lemma can be proved arguing exactly as in the proof of Lemma~\ref{lemma:est1} and Corollary~\ref{cor:inf}.
	\end{proof}
	
	The proof of the following compactness result is completely analogous to that of Proposition~\ref{prop:compactness}.
	
	\begin{prop}[\textbf{Compactness}]\label{prop:compactness2}
		Assume \ref{hyp:W1}--\ref{hyp:W6}, \ref{hyp:pi1}, and \ref{hyp:pi3}. If $\EE_\eps(y_\eps)\le C\eps^2$ for $\eps\in (0,\eps_0)$, then for any $R_\eps$, $u_\eps$ given by Lemma~\ref{lemma:est2}, we have that, up to subsequences,
		\begin{itemize}
			\item $u_\eps\wto u_0$ weakly in $\mathring{W}^{1,p}(\Omega;\R^n)$ with $u_0\in \mathring{H}^{1}(\Omega;\R^n)$,\smallskip
			\item $R_\eps\to R_0$ with $R_0\in \RR$,
		\end{itemize}
		as $\eps\to 0$. Moreover, $R_0$ is independent of the choice of $R_\eps$
		and $u_0$ is independent up to infinitesimal rigid motions of the form $Ax$, with $A\in \R^{n\times n}_{\rm skew}$.
	\end{prop}
	
	The next proposition is the analog of Proposition~\ref{prop:new}. However, in the present setting, owing to the assumptions \ref{hyp:W6} and \ref{hyp:pi3},
	we can improve the result and show convergence on the whole of $\Omega$.
	
	\begin{prop}\label{prop:new2}
		Let $\hat\pi$ be a function as in Lemma~\ref{lemma:hatpi2} and let $y_\eps\in \mathring{W}^{1,p}(\Omega;\R^n)$ satisfy \eqref{eq:estdet2}. Assume there exist $R_\eps\in SO(n)$ converging to $R_0\in\RR$ such that the corresponding displacements $u_\eps$, defined as in \eqref{eq:displacement}, weakly converge in $\mathring{W}^{1,p}(\Omega;\R^n)$ to $u_0\in\mathring{H}^{1}(\Omega;\R^n)$. Then,
		\begin{equation*}
			\lim_{\eps\to 0} \frac1\eps \int_\Omega \left(\hat\pi(y_\eps)\det\nabla y_\eps-\hat\pi(y_{R_\eps})\right)\d x 
			= \int_{\partial\Omega}\pi(R_0x)n_{\partial\Omega}(x)\cdot u_0(x)\d\mc H^{n-1}(x).
		\end{equation*}
	\end{prop}
	
	\begin{proof}
		We follow the lines of the proof of Proposition~\ref{prop:new}. 
		The only difference is in the analysis of the term $I_\eps$ in \eqref{eq:I}, which now can be written as
		\begin{align*}
			I_\eps&= \int_{G_\eps}\hat\pi (y_\eps)\div u_\eps\d x+\sum_{k=2}^n \eps^{k-1}\int_{G_\eps}\hat\pi (y_\eps)\iota_k(\nabla u_\eps)\d x\\
			&\quad\, +\frac1\eps\int_{\Omega_\eps^-\setminus G_\eps}\hat\pi(y_\eps)(\det\nabla y_\eps-1) \d x+\frac1\eps \int_{\Omega_\eps^+}\hat\pi (y_\eps)(\det\nabla y_\eps-1)\d x \\
			&=:\tilde{J}^1_\eps+\tilde{J}^2_\eps+\tilde{J}^3_\eps+\tilde{J}^4_\eps,
		\end{align*}
		where the sets $\Omega_\eps^{\pm}$ and $G_\eps$ are defined as in \eqref{eq:omegapiumeno} and \eqref{eq:goodset}. Here we recall that $G_\eps\subset \Omega^-_\eps$ for $\eps$ small enough.
		The first integral $\tilde{J}^1_\eps$ can be handled exactly as in the proof of Proposition~\ref{prop:new}. 
		To conclude it is enough to show that the remaining terms are infinitesimal, as $\eps\to0$. 
		
		By \eqref{eq:estik} and the Lipschitz continuity of $\hat\pi$ we obtain
		\begin{align*}
			|\tilde J^2_\eps|&\le C\sum_{k=2}^n\eps^{k-1}\int_{G_\eps}(1+\eps|u_\eps|)|\nabla u_\eps|^k\d x\\
			&=C \sum_{k=2}^n\left(\eps^{k-1}\int_{G_\eps}|\nabla u_\eps|^{k-p}|\nabla u_\eps|^p\d x+\eps^k\int_{G_\eps}|u_\eps||\nabla u_\eps|^k\d x\right).
		\end{align*}
		Since $|\nabla u_\eps|\le \eps^{-1/2}$ on $G_\eps$, we have that
		\begin{align*}
			|\tilde{J}^2_\eps|&\le C\sum_{k=2}^n\left(\eps^\frac{k+p-2}{2}\|\nabla u_\eps\|_p^p+\eps^\frac k2 \|u_\eps\|_p\right)\le C \sum_{k=2}^n\left(\eps^\frac{k+p-2}{2}+\eps^\frac k2\right)\le C\eps^\frac p2\to 0.
		\end{align*}
		
		Using  \eqref{eq:estdet2}, the boundedness of $(u_\eps)$ in $W^{1,p}(\Omega;\R^n)$, and recalling that $|\det\nabla y_\eps{-}1|\leq 1$ on $\Omega_\eps^-$, we deduce that
		\begin{align*}
			|\tilde{J}^3_\eps|&\le \frac{C}\eps \int_{\Omega_\eps^-\setminus G_\eps}(1+\eps|u_\eps|)|\det\nabla y_\eps{-}1| \d x\\
			& \le  \frac{C}\eps \|\det\nabla y_\eps{-}1\|_{L^2(\Omega_\eps^-)}|\Omega\setminus G_\eps|^\frac 12+ C\int_{\Omega_\eps^-\setminus G_\eps}|u_\eps|\d x\\
			&\le C |\Omega\setminus G_\eps|^\frac 12+C\|u_\eps\|_p|\Omega\setminus G_\eps|^\frac{1}{p'}\le C |\Omega\setminus G_\eps|^\frac 12+C|\Omega\setminus G_\eps|^\frac{1}{p'},
		\end{align*}
		where the last term goes to zero, as $\eps\to 0$, by \eqref{eq:goodsetvanish}.
		
		To deal with $\tilde{J}^4_\eps$ we consider the two cases $q=1$ and $q\in(1,2]$, separately. If $q=1$, the function $\hat\pi$ is bounded and so, we have
		\begin{equation*}
			|\tilde{J}^4_\eps|\le \frac{C}\eps\|\det\nabla y_\eps-1\|_{L^1(\Omega_\eps^+)}\le C\eps,
		\end{equation*}
		where in the last inequality we used \eqref{eq:estdet2}.
		
		If instead $q\in (1,2]$, by using the $p/q'$-growth of $\hat\pi$ we have
		\begin{align*}
			|\tilde{J}^4_\eps|&\le \frac{C}\eps \int_{\Omega_\eps^+}(1+(\eps|u_\eps|)^\frac{p}{q'}) |\det\nabla y_\eps-1| \d x\\
			&\le  \frac{C}\eps \|\det\nabla y_\eps-1\|_{L^q(\Omega_\eps^+)} |\Omega_\eps^+|^\frac{1}{q'}+C\eps^{\frac{p}{q'}-1}\|\det\nabla y_\eps-1\|_{L^q(\Omega_\eps^+)}\|u_\eps\|_p^\frac{p}{q'}.
		\end{align*}
		By \eqref{eq:estdet2} and the boundedness of $(u_\eps)$ in $W^{1,p}(\Omega;\R^n)$ we deduce
		\begin{align*}
			|\widetilde{J}^4_\eps|&\le C \eps^{\frac{2}{q}-1}|\Omega^+_\eps|^\frac{1}{q'}+ C \eps^{\frac{2}{q}+\frac{p}{q'}-1}\|u_\eps\|_p^\frac{p}{q'}\le C |\Omega^+_\eps|^\frac{1}{q'}+ C \eps^{\frac{2}{q}+\frac{p}{q'}-1}.
		\end{align*}
		It is easy to verify that $\frac{2}{q}+\frac{p}{q'}-1>0$. Moreover, since $G_\eps\subset\Omega_\eps^-$ for $\eps$ small enough, we have that $|\Omega_\eps^+|\to0$ by \eqref{eq:goodsetvanish}. Therefore, we conclude that in both cases $\widetilde{J}^4_\eps$ is infinitesimal, as $\eps\to 0$.
	\end{proof}
	
	The next proposition collects both the liminf and the limsup inequalities, which thus provide a full $\Gamma$-convergence result.
	
	\begin{prop}[\textbf{Liminf and limsup inequalities}]\label{prop:gammaliminf2}
		Assume \ref{hyp:W1}--\ref{hyp:W6}, \ref{hyp:pi1}, and \ref{hyp:pi3}. For every $\eps\in(0,\eps_0)$ let $y_\eps\in \mathring{W}^{1,p}(\Omega;\R^n)$ be such that there exist $R_\eps\in SO(n)$ converging to $R_0\in\RR$ and the corresponding displacements $u_\eps$, defined as in \eqref{eq:displacement}, weakly converge in $\mathring{W}^{1,p}(\Omega;\R^n)$ to $u_0\in\mathring{H}^{1}(\Omega;\R^n)$. Then		
		\begin{equation*}
			\EE_0(u_0,R_0)\le \liminf\limits_{\eps\to 0}\frac{1}{\eps^2}\EE_\eps(y_\eps),
		\end{equation*}
		where $\EE_0$ is the functional defined in \eqref{eq:E0}.
		
		On the other hand, for every $(u_0,R_0)\in \mathring{H}^{1}(\Omega;\R^n)\times\RR$ there exist $(u_\eps,R_\eps)\in \mathring{W}^{1,p}(\Omega;\R^n)\times SO(n)$ such that $u_\eps\wto u_0$ weakly in $\mathring{W}^{1,p}(\Omega;\R^n)$, $R_\eps\to R_0$ and, setting $y_\eps(x):=R_\eps(x+\eps u_\eps(x))$, there holds
		\begin{equation*}	
			\limsup\limits_{\eps\to 0}\frac{1}{\eps^2}\EE_\eps(y_\eps)\le \EE_0(u_0,R_0).
		\end{equation*}
	\end{prop}
	
	\begin{proof}
		The proof is the same of Propositions~\ref{prop:gammaliminf} and~\ref{prop:gammalimsup}, once we have at our disposal Proposition~\ref{prop:new2}.
		The only additional remark is that the deformations $y_\eps$ in the recovery sequence are admissible since $\det\nabla y_\eps\in L^\infty(\Omega)$ by construction (see \eqref{eq:Xpq} for the definition of~$Y^p_q$). 
	\end{proof}
	
	Combining the previous results and arguing as in Theorem~\ref{thm:convmin}, one can infer the following convergence result for almost minimizers.
	
	\begin{thm}[\textbf{Convergence of almost minimizers}]\label{thm:convmin2}
		Assume \ref{hyp:W1}--\ref{hyp:W6}, \ref{hyp:pi1}, and \ref{hyp:pi3}.
		If $(y_\eps)$ is a sequence of almost minimizers for the energies $\EE_\eps$, that is, 
		\begin{equation*}
			\EE_\eps(y_\eps)\le \inf\limits_{\mathring{W}^{1,p}(\Omega;\R^n)}\EE_\eps +o(\eps^2), 
		\end{equation*}
		then there exist $R_\eps\in SO(n)$ such that, up to passing to a subsequence, we have
		\begin{itemize}
			\item $u_\eps\to u_0$ strongly in $\mathring{W}^{1,p}(\Omega;\R^n)$ with $u_0\in \mathring{H}^{1}(\Omega;\R^n)$;\smallskip
			\item $R_\eps\to R_0$ with $R_0\in \RR$.
		\end{itemize}
		Furthermore, the pair $(u_0,R_0)$ is a minimizer of $\EE_0$ on $\mathring{H}^{1}(\Omega;\R^n)\times \RR$ and
		\begin{equation*}
			\lim\limits_{\eps\to 0} \frac{1}{\eps^2} \Big(\inf\limits_{\mathring{W}^{1,p}(\Omega;\R^n)}\EE_\eps\Big) =\min
			\big\{ \EE_0(u,R) : \ (u,R)\in\mathring{H}^{1}(\Omega;\R^n)\times \RR \big\}.
		\end{equation*}
	\end{thm}

	\section{A refined $\Gamma$-limit and a comparison with dead loads}\label{sec:secondorder}
	
	In this section we make a comparison with the results obtained in \cite{MaorMora}, in the case of dead loads. In particular, we compute a refined version of the $\Gamma$-limit of the rescaled energies $\frac 1{\eps^2}\mc E_\eps$.
	
	We assume \ref{hyp:W1}--\ref{hyp:W5} and \ref{hyp:pi1}, together with either \ref{hyp:pi2} or \ref{hyp:pi3} and \ref{hyp:W6}. In addition, we require\smallskip
	\begin{enumerate}[label=\textup{($\pi$4)}]
		\item\label{hyp:pi4} $\pi$ is of class $C^2$ in the open set given by \ref{hyp:pi1}.\smallskip
	\end{enumerate}
	Under this assumption any optimal rotation $R_0\in \RR$ satisfies, in addition to \eqref{eq:EL}, the following condition:
	\begin{equation}\label{eq:2variation}
		\int_{\Omega}\big(\nabla\pi (R_0x)\cdot R_0 A^2x+D^2\pi(R_0 x)R_0 Ax\cdot R_0 Ax \big)\d x\ge0\qquad\text{for every }A\in \R^{n\times n}_{\rm skew}.
	\end{equation}
	This is obtained by imposing that the second variation of the functional in \eqref{eq:optrot} is positive semidefinite along the curve $t\mapsto R_0e^{tA}$.
	By the Divergence Theorem, condition \eqref{eq:2variation} can be rewritten as
	\begin{equation}\label{eq:2variation2}
		\int_{\partial\Omega}\big(\nabla\pi (R_0x)\cdot R_0Ax\big)n_{\partial\Omega}(x)\cdot Ax\d\mc H^{n-1}(x)\ge 0\qquad\text{for every }A\in \R^{n\times n}_{\rm skew}.
	\end{equation}
	
	In \cite{MaorMora} the applied body force is assumed to be a dead load of the form
	\begin{equation*}
		-\int_\Omega g(x)\cdot y(x) \d x,
	\end{equation*}
	where $g\in L^2(\Omega;\R^n)$ is given. In this setting the authors proved that the set of optimal rotations, which is defined as
	\begin{equation}\label{eq:Rg}
		\RR_g:=\argmin\limits_{R\in SO(n)}\left\{-\int_\Omega g(x)\cdot Rx\d x\right\},
	\end{equation} 
	is a submanifold of $SO(n)$ (see \cite[Proposition~4.1]{MaorMora}). Moreover, if $(y_\eps)$ is a sequence of deformations with total energy of order $\eps^2$, 
	then any sequence of rotations $(R_\eps)$ provided by the rigidity estimate converges to an optimal rotation $R_0$ (as in Propositions~\ref{prop:compactness} and~\ref{prop:compactness2}) and, in addition, satisfies
	\begin{equation}\label{eq:distsqrt}
		\dist_{SO(n)}(R_\eps;\RR_g)\leq C\sqrt\eps,
	\end{equation}
	where $\dist_{SO(n)}$ is the intrinsic distance in $SO(n)$, that is,
	\begin{equation*}
		\dist_{SO(n)}(R,S):=\min\big\{ |A| : \ A\in\R^{n\times n}_{\rm skew},\ R=Se^A\big\}.
	\end{equation*}
	Finally, the $\Gamma$-limit of the rescaled energies can be expressed as
	\begin{equation}\label{eq:oldG}
		\frac 12\int_\Omega Q(x,e(u_0)(x))\d x -\int_\Omega g(x)\cdot R_0u_0(x) \d x -\frac12 \int_\Omega g(x)\cdot R_0A_0^2x \d x,
	\end{equation}
	where $u_0$ is the limit displacement, $A_0\in\R^{n\times n}_{\rm skew}$ is the limit of 
	\begin{equation*}
		\frac 1{\sqrt\eps} R_0^T(R_\eps -\mc P_{\!g}(R_\eps))
	\end{equation*}
	(which exists up to subsequences), and $\mc P_{\!g}$ is the projection operator on $\RR_g$ (see \cite[Section~5]{MaorMora}).
	We note that the last term in \eqref{eq:oldG} is the second variation of the functional in \eqref{eq:Rg} at $R_0$ computed in the direction $A_0$.
	In \cite{MaorMora} it is then proved that $A_0=0$ for sequences $(y_\eps)$ of almost minimizers, so that the last term in \eqref{eq:oldG} is identically equal to $0$ on minimizers.
	
	In our setting of a pressure live load, a first difference with \cite{MaorMora} is that the set of optimal rotations may not be a manifold, as the following example shows.
	
	\begin{ex}\label{ex:ex}
		Let $n=2$ and let $\Omega$ be the set given in polar coordinates by
		\begin{equation*}
			\Omega=\{(\rho,\theta):\ \rho<1 \text{ if }\theta\in[0,\pi/2]\cup[\pi,3\pi/2],\text{ and }\rho<2\text{ otherwise}\}
		\end{equation*}
		\begin{figure}
			\begin{tikzpicture}
				\filldraw[color=black!5, fill=black!5] (0,0) -- (1,0) arc (0:90:1);
				\filldraw[color=black!5, fill=black!5] (0,0) -- (-1,0) arc (180:270:1);
				\filldraw[color=black!5, fill=black!5] (0,0) -- (0,2) arc (90:180:2);
				\filldraw[color=black!5, fill=black!5] (0,0) -- (0,-2) arc (270:360:2);
				\draw[thick] (1,0) arc [start angle=0, end angle=90, radius=1cm];
				\draw[thick] (-1,0) arc [start angle=180, end angle=270, radius=1cm];
				\draw[thick] (0,2) arc [start angle=90, end angle=180, radius=2cm];
				\draw[thick] (0,-2) arc [start angle=270, end angle=360, radius=2cm];
				\draw[->]  (2,0) -- (3,0);
				\node [below] at (3,0) {\small$x_1$};
				\draw[thick]  (1,0) -- (2,0);
				\draw (-3,0) -- (-2,0);
				\draw[thick] (-2,0) -- (-1,0);
				\draw[dashed]  (-1,0) -- (1,0);
				\draw[->]  (0,2) -- (0,3);
				\node [left] at (0,3) {\small$x_2$};
				\draw[thick]  (0,1) -- (0,2);
				\draw (0,-3) -- (0,-2);
				\draw[thick] (0,-2) -- (0,-1);
				\draw[dashed]  (0,-1) -- (0,1);
				%\node [above] at (1,-1) {\small$\Omega$};
			\end{tikzpicture}
			\caption{The set $\Omega$ in Example~\ref{ex:ex}.}\label{figure}
		\end{figure}
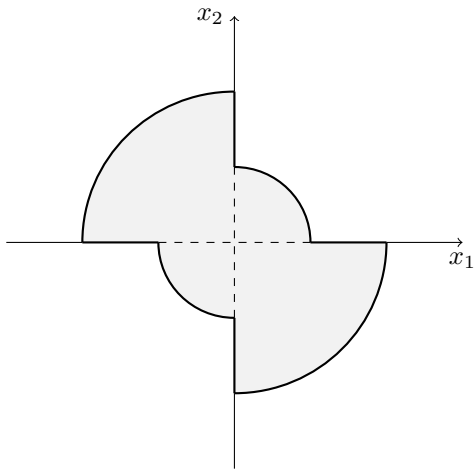
		(see Figure~\ref{figure}). We note that \eqref{eq:baricenter} is satisfied. 
		
		Let $\varphi\in C^3([0,\pi/2])$ be a nonnegative function satisfying $\varphi(0)=0$ and attaining its maximum at $\alpha=\pi/2$. Suppose, in addition, that the first, second, and third derivatives of $\varphi$ vanish at $\alpha=0$ and at $\alpha=\pi/2$.
		Let $\psi\in C^2([1,+\infty))$ be a bounded function with bounded first and second derivatives satisfying $\psi(1)=\psi'(1)=\psi''(1)=0$ and $\int_{1}^{2}\rho\psi(\rho)\d\rho=1$. A simple choice of $\psi$ could be $\psi(\rho)=\frac{20}{9}(\rho-1)^3$ for $\rho\in [1,2]$, suitably extended to $[2,+\infty)$.
		
		We consider the following pressure function:
		\begin{equation*}
			\pi(x_1,x_2)=\begin{cases}
				\psi(\sqrt{x_1^2+x_2^2})\varphi'(\arctan \frac{x_2}{x_1}) &\text{ if }x_1>0,x_2>0,\text{ and }x_1^2+x_2^2>1,\smallskip\\
				0 &\text{ otherwise.}
			\end{cases}
		\end{equation*}
		Elementary computations show that $\pi$ is Lipschitz continuous and of class $C^2(\R^n)$, so that \ref{hyp:pi1} and \ref{hyp:pi4} are satisfied; furthermore, $\pi$ is bounded, so also \ref{hyp:pi3} holds true.
		
		We recall that $SO(2)$ can be identified with the unit sphere $\mathbb{S}^1$ via the map 
		\begin{equation*}
			[0,2\pi)\ni\alpha\mapsto R^\alpha:=\begin{pmatrix}
				\cos\alpha& -\sin\alpha\\ \sin\alpha& \cos\alpha
			\end{pmatrix},
		\end{equation*}
		so that we can write
		\begin{equation*}
			\RR=\argmin\limits_{\alpha\in[0,2\pi)}\left\{\int_{R^\alpha\Omega}\pi\left(x_1,x_2\right)\d x_1\d x_2\right\}.
		\end{equation*}
		For $\alpha\in [0,\pi/2)$ we can compute
		\begin{align*}
			\int_{R^\alpha\Omega}\pi\left(x_1,x_2\right)\d x_1\d x_2=\int_{1}^{2}\int_{0}^{\alpha}\psi(\rho)\varphi'(\theta)\rho\d\theta\d\rho=\int_{1}^{2}\rho\psi(\rho)\d\rho\int_{0}^{\alpha}\varphi'(\theta)\d\theta=\varphi(\alpha).
		\end{align*}
		Similarly, we obtain
		\begin{equation}\label{eq:phi}
			\int_{R^\alpha\Omega}\pi\left(x_1,x_2\right)\d x_1\d x_2 =\phi(\alpha):=\begin{cases}
				\varphi(\alpha)&\text{ if }\alpha\in [0,\pi/2),\\
				\varphi(\pi/2)-\varphi(\alpha-\pi/2) &\text{ if }\alpha\in [\pi/2,\pi),\\
				\varphi(\alpha-\pi) &\text{ if }\alpha\in [\pi,3\pi/2),\\
				\varphi(\pi/2)-\varphi(\alpha-3\pi/2) &\text{ if }\alpha\in [3\pi/2,2\pi).
			\end{cases}
		\end{equation}
		Since $\varphi$ is nonnegative and maximized at $\pi/2$, the set of optimal rotations corresponds to the zero-level set of the function $\phi$. In particular, in $[0,\pi/2)$
		this is given by the zero-level set of $\varphi$, which can be any closed set at positive distance from $\pi/2$.
		
		This example shows that, in general, we cannot expect $\RR$ to be a manifold.
	\end{ex}
	
	Since $\RR$ is not in general a manifold, the projection operator on $\RR$ is not well defined. However, in the limiting process we can keep track of the distance of the approximating rotations $R_\eps$ from $\RR$ through a suitable sequence of skew-symmetric matrices $A_\eps$. In contrast with \eqref{eq:distsqrt}, the scaling of this distance may be larger than $\sqrt\eps$ (actually, larger than $\sqrt[k]{\eps}$ for any given $k>2$), see Example~\ref{rmk:exsqrt}. 
	To recover compactness of $(A_\eps)$ we rescale it by $|A_\eps|\vee \sqrt{\eps}$ and we denote by $A_0$ its limit. The $\Gamma$-limit of the rescaled energies can be then expressed as
	\begin{equation*}
		\EE_0(u_0,R_0)+\frac 12\FF(R_0,A_0),
	\end{equation*}
	where $\FF\colon \RR\times \R^{n\times n}_{\rm skew}\to [0,+\infty)$ is the second variation of the functional in \eqref{eq:optrot}.
	This additional term measures the cost due to the fluctuations of the approximating rotations from the set $\RR$.
	Arguing as in \eqref{eq:2variation} and \eqref{eq:2variation2}, the functional $\mc F$ takes the form
	\begin{equation*}
		\FF(R_0,A_0)=\int_{\partial\Omega}\big(\nabla\pi (R_0x)\cdot R_0 A_0x\big)A_0x\cdot n_{\partial\Omega}(x)\d\mc H^{n-1}(x).
	\end{equation*}
	For sequences $(y_\eps)$ of almost minimizers the limit $A_0$ may be different from $0$; however, we have $\FF(R_0,A_0)=0$.
	More precisely, we have the following result.
	
	\begin{thm}\label{thm:Gammaref}
		Under the assumptions of Proposition~\ref{prop:compactness} or Proposition~\ref{prop:compactness2}, we have in addition that there exist $A_\eps\in \R^{n\times n}_{\rm skew}$ such that $R_\eps=S_\eps e^{A_\eps}$ for some $S_\eps\in\RR$ and, up to subsequences, $A_\eps\to 0$ and $\frac{A_\eps}{|A_\eps|\vee \sqrt{\eps}}\to A_0$ for some $A_0\in \R^{n\times n}_{\rm skew}$ with $|A_0|\le 1$.
		
		If also \ref{hyp:pi4} is in force, then
		\begin{equation*}
			\Gamma-\lim\limits_{\eps\to 0}\frac{1}{\eps^2}\EE_\eps(y_\eps)= \EE_0(u_0,R_0)+\frac12\FF(R_0,A_0)
		\end{equation*}
		with respect to the following convergences: $u_\eps\wto u_0$ weakly in $\mathring{W}^{1,p}(\Omega;\R^n)$, $R_\eps\to R_0$, $A_\eps\to 0$, and $\frac{A_\eps}{|A_\eps|\vee \sqrt{\eps}}\to A_0$.
		
		Finally, if $(y_\eps)$ is a sequence of almost minimizers for the energies $\EE_\eps$, that is, \eqref{eq:almmin} holds,
		then there exist $R_\eps\in SO(n)$ and $A_\eps\in \R^{n\times n}_{\rm skew}$ as above
		such that, up to a subsequence, we have
		\begin{itemize}
			\item $u_\eps\to u_0$ strongly in $\mathring{W}^{1,p}(\Omega;\R^n)$ with $u_0\in \mathring{H}^{1}(\Omega;\R^n)$,\smallskip
			\item $R_\eps\to R_0$ with $R_0\in \RR$,\smallskip
			\item $A_\eps\to 0$ and $\frac{A_\eps}{|A_\eps|\vee \sqrt{\eps}}\to A_0$ with $A_0\in \R^{n\times n}_{\rm skew}$, $|A_0|\le 1$.
		\end{itemize}
		Furthermore, the triplet $(u_0,R_0,A_0)$ is a minimizer of $\EE_0+\frac12\FF$ on $\mathring{H}^{1}(\Omega;\R^n)\times \RR\times  \R^{n\times n}_{\rm skew}$, $\FF(R_0,A_0)=0$, and
		\begin{equation*}
			\begin{aligned}
				\lim\limits_{\eps\to 0}\frac{1}{\eps^2} \Big(\inf\limits_{\mathring{W}^{1,p}(\Omega;\R^n)}\EE_\eps\Big)
				& =\min \bigg\{ \EE_0(u,R)+\frac12\FF(R,A): \ (u,R,A)\in \mathring{H}^{1}(\Omega;\R^n)\times \RR\times  \R^{n\times n}_{\rm skew}\bigg\}
				\\
				& =\min \big\{ \EE_0(u,R): \ (u,R)\in \mathring{H}^{1}(\Omega;\R^n)\times \RR \big\}.
			\end{aligned}
		\end{equation*}
	\end{thm}
	
	\begin{proof}
		As for the compactness statement, since $\RR$ is a closed set, there exist $S_\eps\in\RR$ and $A_\eps\in \R^{n\times n}_{\rm skew}$ such that 
		$R_\eps=S_\eps e^{A_\eps}$ and
		\begin{equation*}
			\dist_{SO(n)}(R_\eps;\RR)=\dist_{SO(n)}(R_\eps, S_\eps)=|A_\eps|. 
		\end{equation*}
		Since, up to subsequences, $R_\eps\to R_0\in \RR$, we have that $A_\eps\to 0$. The remaining properties follow from the fact the sequence $\Big(\frac{A_\eps}{|A_\eps|\vee \sqrt{\eps}}\Big)$ is bounded by $1$.
		
		We now give a sketch of the proof of the liminf inequality.
		By Proposition~\ref{prop:gammaliminf} or \ref{prop:gammaliminf2} we have  that 
		\begin{equation*}
			\liminf\limits_{\eps\to 0}\frac{1}{\eps^2}\EE_\eps(y_\eps)\ge \EE_0(u_0,R_0)+\liminf\limits_{\eps\to 0}\int_{\Omega}\frac{\pi(R_\eps x)-\pi (x)}{\eps}\d x.
		\end{equation*}
		In fact, the last term above was always neglected in the previous computations, since it is nonnegative by \eqref{eq:IinR}. 
		Using that $S_\eps\in\RR$, a Taylor expansion of $\pi$ and of the exponential map yields
		\begin{align}
			\int_{\Omega}\frac{\pi(R_\eps x)-\pi (x)}{\eps}\d x
			&=\frac{1}{\eps}\int_{\Omega}\nabla\pi(S_\eps x)\cdot S_\eps {A}_\eps x\d x \nonumber \\
			&\ +\frac{1}{2\eps} \int_{\Omega}\big( \nabla\pi(S_\eps x)\cdot S_\eps {A}^2_\eps x+ D^2\pi(S_\eps(I+t_\eps(x) K_\eps) x)S_\eps K_\eps x\cdot S_\eps K_\eps x\big)\d x \nonumber \\
			&\ +\frac{1}{\eps}\int_{\Omega}|A_\eps|^3\nabla\pi(S_\eps x)\cdot S_\eps H_\eps x\d x, \label{eq:align}
		\end{align}
		where $t_\eps (x)\in [0,1]$, $H_\eps$ is a uniformly bounded matrix, and
		\begin{equation*}
			K_\eps:={A}_\eps+\frac{1}{2}{A}_\eps^2+|A_\eps|^3 H_\eps.
		\end{equation*}
		We now note that the first integral in \eqref{eq:align} is equal to $0$ by \eqref{eq:EL}; thus, by multiplying and dividing by $\alpha_\eps:=(|A_\eps|\vee\sqrt{\eps})^2$ we obtain
		\begin{align*}
			\int_{\Omega}&\frac{\pi(R_\eps x)-\pi (x)}{\eps}\d x \\
			&=\frac{\alpha_\eps}{2\eps}\Bigg[ \int_{\Omega}\Big( \nabla\pi(S_\eps x)\cdot S_\eps \frac{A^2_\eps x}{\alpha_\eps} + D^2\pi(S_\eps(I{+}t_\eps(x) K_\eps)x )S_\eps \frac{K_\eps x}{\sqrt{\alpha_\eps}}\cdot S_\eps \frac{K_\eps x}{\sqrt{\alpha_\eps}}\Big) \d x\\
			&\qquad\qquad\qquad\qquad+\frac{2|A_\eps|^3}{\alpha_\eps}\int_{\Omega}\nabla\pi(S_\eps x)\cdot S_\eps H_\eps x\d x\Bigg].
		\end{align*}
		We observe that the left-hand side is nonnegative, hence the term within square brackets is nonnegative, as well.
		Since $\sqrt{\alpha_\eps}=|A_\eps|\vee\sqrt{\eps}\ge \sqrt{\eps}$, $|A_\eps|\to 0$, $S_\eps\to R_0$, $\frac{A_\eps}{\sqrt{\alpha_\eps}}\to A_0$, letting $\eps\to 0$ and using the Dominated Convergence Theorem yield
		\begin{align*}
			\liminf\limits_{\eps\to 0}\int_{\Omega}\frac{\pi(R_\eps x)-\pi (x)}{\eps}\d x\ge \frac 12\int_{\Omega}\big(\nabla\pi (R_0x)\cdot R_0 A_0^2x+D^2\pi(R_0 x)R_0 A_0 x\cdot R_0 A_0 x\big)\d x.
		\end{align*}
		The liminf inequality follows now from the Divergence Theorem.
		
		For the construction of the recovery sequence we proceed as in Proposition~\ref{prop:gammalimsup} or~\ref{prop:gammaliminf2}, but choosing $R_\eps:=R_0e^{A_\eps}$ with
		$A_\eps:=\sqrt{\eps}A_0$. Since $R_0$ is in $\RR$, we can write
		\begin{equation}\label{eq:split}
			\begin{aligned}
				\frac {1}{\eps^2}\EE_\eps(y_\eps)&=\frac{1}{\eps^2}\int_{\Omega}W(x,\nabla y_\eps)\d x+\frac 1\eps\int_{\Omega}\left(\pi (y_\eps(x))\det\nabla y_\eps-\pi(R_\eps x)\right)\d x\\
				&\quad+\int_{\Omega}\frac{\pi(R_\eps x)-\pi (R_0 x)}{\eps}\d x.
			\end{aligned}
		\end{equation}
		The first two integrals at the right-hand side can be bounded by $\EE_0(u_0,R_0)$, arguing as in Proposition~\ref{prop:gammalimsup} or~\ref{prop:gammaliminf2}.
		Repeating the same computations as for the liminf inequality, one can show that
		the last integral converges to $\frac12\FF(R_0,A_0)$. 
		
		Convergence of almost minimizers and of infima can be proved exactly as in Theorems~\ref{thm:convmin} and~\ref{thm:convmin2}.
		Finally, we observe that the functional $\FF$ is always nonnegative by \eqref{eq:2variation2} and $\FF(R,0)=0$ for every $R\in\RR$. Hence, by minimality we deduce that $\FF(R_0,A_0)=0$. This concludes the proof.
	\end{proof}
	
	We conclude the paper with an example showing that, given any sequence $\lambda_\eps\to0$ such that 
	\begin{equation}\label{eq:lambda}
		\lim\limits_{\eps\to 0}\frac{\lambda_\eps^2}{\eps}=+\infty\qquad\text{ and }\qquad\lim\limits_{\eps\to 0}\frac{\lambda_\eps^k}{\eps}=0
	\end{equation}
	for some $k>2$, there may exist sequences of almost minimizers of $\mc E_\eps$ for which any approximating sequence of rotations has a distance from $\RR$ of order $\lambda_\eps$. This provides a further difference with the case of dead loads \cite{MaorMora}.

	\begin{ex}\label{rmk:exsqrt}
		We start by considering a sequence $(\lambda_\eps)$ satisfying \eqref{eq:lambda} with $k=3$.
		We assume the pressure intensity $\pi$ to be of class $C^3$. Moreover, we assume that the set of optimal rotations $\RR$ is finite and that
		for every $R_0\in\RR$ there exists $A_0\in \R^{n\times n}_{\rm skew}$ such that 
		\begin{equation}\label{eq:propexample}
			|A_0|=1 \qquad \text{and} \qquad \FF(R_0,A_0)=0. 
		\end{equation}
		These properties are satisfied, for instance, in Example~\ref{ex:ex} if the function $\varphi$ is strictly increasing. Indeed, in this case the function $\phi$ in \eqref{eq:phi} attains its minimum only at $\alpha=0$ and $\alpha=\pi$ and thus, $\RR=\{\pm I\}$. Moreover, in this example $\nabla\pi(x)=\nabla\pi(-x)=0$ for every $x\in\partial\Omega$, so that $\FF(R_0,A_0)=0$ for every $R_0\in\RR$ and every $A_0\in \R^{n\times n}_{\rm skew}$.
		Finally, $\pi$ is of class $C^3$ if we assume, in addition, $\varphi\in C^4([0,\pi/2])$ with $\varphi^{\rm (iv)}(0)=\varphi^{\rm (iv)}(\pi/2)=0$ and $\psi\in C^3([1,+\infty))$ with bounded third derivative and $\psi'''(1)=0$.
		
		Now let $(u_0, R_0)$ be a minimizer of $\EE_0$ on $\mathring{H}^{1}(\Omega;\R^n)\times \RR$ and let $A_0\in \R^{n\times n}_{\rm skew}$ satisfy \eqref{eq:propexample}.
		Let $R_\eps:=R_0 e^{\lambda_\eps A_0}$ and let $(u_\eps)$ be an approximating sequence for $u_0$ as in \eqref{eq:strong}.
		We claim that the deformations $y_\eps(x):=R_\eps(x+\eps u_\eps(x))$ are a sequence of almost minimizers of $\EE_\eps$. Indeed, 
		arguing as in the proof of the limsup inequality in Theorem~\ref{thm:Gammaref},
		we have by \eqref{eq:split} that
		\begin{align*}
			\lim\limits_{\eps\to 0}\frac{1}{\eps^2}\EE_\eps(y_\eps)&=\EE_0(u_0,R_0)+\lim\limits_{\eps\to 0} \int_{\Omega}\frac{\pi(R_\eps x){-}\pi (R_0 x)}{\eps}\d x= \min \limits_{\mathring{H}^{1}\times \RR}\EE_0+\lim\limits_{\eps\to 0} \int_{\Omega}\frac{\pi(R_\eps x){-}\pi (R_0 x)}{\eps}\d x\\
			&= \lim\limits_{\eps\to 0}\left(\frac{1}{\eps^2}\inf\limits_{\mathring{W}^{1,p}(\Omega;\R^n)}\EE_\eps+\int_{\Omega}\frac{\pi(R_\eps x){-}\pi (R_0 x)}{\eps}\d x\right).
		\end{align*}
		Therefore, the claim is proved if we show that the last term above vanishes, as $\eps\to 0$. To prove it we argue as in the proof of the liminf inequality in Theorem~\ref{thm:Gammaref}, now expanding up to the third order. By \eqref{eq:EL} and \eqref{eq:propexample} we obtain
		\begin{align}\label{eq:Taylor3}
			\int_{\Omega}\frac{\pi(R_\eps x){-}\pi (R_0 x)}{\eps}\d x=\frac{\lambda_\eps}{\eps}\int_{\Omega}\nabla\pi(R_0 x)\cdot R_0 {A}_0 x\d x+\frac{\lambda_\eps^2}{2\eps}\FF(R_0,A_0)+O\left(\frac{\lambda_\eps^3}{\eps}\right)=O\left(\frac{\lambda_\eps^3}{\eps}\right),
		\end{align}
		which proves the claim owing to \eqref{eq:lambda} with $k=3$.
		
		We note that the intrinsic distance of $R_\eps=R_0 e^{\lambda_\eps A_0}$ from $\RR$ in $SO(n)$ is of order $\lambda_\eps$. Indeed, since $\RR$ is finite by assumption,
		we have that $\d_{SO(n)}(R_\eps;\RR)=\d_{SO(n)}(R_\eps,R_0)$ for $\eps$ small enough. 
		By definition we clearly have $\d_{SO(n)}(R_\eps,R_0)\leq \lambda_\eps$. On the other hand, since the intrinsic distance in $SO(n)$ is equivalent to the Euclidean distance,
		we obtain
		\begin{equation*}
			\d_{SO(n)}(R_\eps,R_0)\ge c|e^{\lambda_\eps A_0}-I|\ge c \lambda_\eps.
		\end{equation*}
		
		We now prove that for the sequence $(y_\eps)$ constructed above, any sequence $(R_\eps')$ of approximating rotations satisfies
		\begin{equation}\label{eq:absurd}
			\d_{SO(n)}(R_\eps';\RR)\ge c \lambda_\eps.
		\end{equation}
		By Lemma~\ref{lemma:est1} or Lemma~\ref{lemma:est2} we deduce that $|R_\eps-R'_\eps|\le C\eps$, hence $R_\eps'\to R_0$. 
		Since $\RR$ is finite, we have that $\d_{SO(n)}(R_\eps';\RR)=\d_{SO(n)}(R_\eps',R_0)$ for $\eps$ small enough. 
		Let $A_\eps'\in \R^{n\times n}_{\rm skew}$ be such that $R_\eps'=R_0 e^{A_\eps'}$ and $\d_{SO(n)}(R_\eps',R_0)=|A_\eps'|$. Assume by contradiction that \eqref{eq:absurd} does not hold, that is,
		$|A'_\eps|/\lambda_\eps\to0$, as $\eps\to0$. Then we have
		\begin{equation*}
			C\eps\ge |R_\eps-R_\eps'| = | e^{\lambda_\eps A_0}-e^{A_\eps'}| 
			\ge |e^{\lambda_\eps A_0}-I | - | e^{A'_\eps}-I | \ge |e^{\lambda_\eps A_0}-I | - c|A_\eps'|.
		\end{equation*}
		Dividing by $\lambda_\eps$ and sending $\eps\to0$, we obtain a contradiction, since the left-hand side vanishes by \eqref{eq:lambda}
		and the right-hand side converges to $|A_0|=1$.
		
		Using again that $|R_\eps-R'_\eps|\le C\eps$, it is easy to see that in fact the intrinsic distance of $R_\eps'$ from $\RR$ is of order $\lambda_\eps$.
		
		If, instead, the sequence $(\lambda_\eps)$ satisfies \eqref{eq:lambda} for some $k\geq4$, the previous arguments can be adapted with small changes as follows. 
		We assume, in addition, that $\pi$ is of class $C^k$ and that for every $R_0\in\RR$ there exists $A_0\in \R^{n\times n}_{\rm skew}$ such that 
		\begin{equation}\label{eq:propexample2}
			|A_0|=1 \qquad \text{and} \qquad {\mc V}_j(R_0,A_0)=0 \quad\text{ for every } j=2,\dots,k-1,
		\end{equation}
		where ${\mc V}_j(R_0,A_0)$ is the $j$-th variation of the functional in \eqref{eq:optrot} at $R_0$ computed in the direction $A_0$.
		This is fulfilled by the pressure load in Example~\ref{ex:ex}, if $\varphi$ and $\psi$ have enough regularity and satisfy suitable boundary conditions.
		By expanding up to order $k$ in \eqref{eq:Taylor3}, condition \eqref{eq:propexample2} guarantees that the sequence $(y_\eps)$, constructed as above, is still a sequence of almost minimizers. The bounds on the intrinsic distance from $\RR$ can be proved as before.
	\end{ex}

	\begin{rmk}
		If condition \eqref{eq:propexample} is not satisfied, that is, for every $R_0\in\RR$ one has 
		\begin{equation*}
			\FF(R_0,A_0)=0 \qquad \Longleftrightarrow \qquad A_0=0, 
		\end{equation*}
		the phenomenon described in the previous example can not arise.
		More precisely, one can show that the intrinsic distance of the approximating rotations from $\RR$ is at most of order $\sqrt{\eps}$, as in \eqref{eq:distsqrt}.
		The argument is the same as in \cite[Theorem~5.1]{MaorMora}, combined with \eqref{eq:align}.
	\end{rmk}
	
	\bigskip
	
	\subsection*{Acknowledgements.}
	The authors acknowledge support by the Italian Ministry of University and Research through the National Research Project PRIN 2017 ``Variational Methods for Stationary and Evolution Problems with Singularities and Interfaces''. The authors are members of GNAMPA--INdAM. 
	\bigskip

	\bibliographystyle{siam}

\end{document}